\documentclass[a4paper,11pt]{article} 
\usepackage{amsmath}
\usepackage{amssymb}
\usepackage{graphicx}
\usepackage{color}
\usepackage{amssymb}
\usepackage{latexsym}
\usepackage{dsfont,pifont}
\usepackage{bbold}
\usepackage{tikz}
\usepackage{pgfplots}
\usepackage{bm,upgreek}
\usepackage{caption}
\usepackage{subcaption}
\usepackage{blkarray}
\usepackage{mathtools}
\usepackage{float}
\usepackage[algoruled]{algorithm2e}

\usepackage{hyperref}

\newcommand{\ud}{\, \mathrm{d}}  
\newcommand{\ue}{e}

\newcommand{\GG}{\mathcal{G}}
\newcommand{\RR}{\mathcal{R}}

\newcommand{\m}[1]{\begin{bmatrix} #1 \end{bmatrix}}
\DeclarePairedDelimiter{\abs}{\lvert}{\rvert}
\DeclarePairedDelimiter{\norm}{\lVert}{\rVert}

\newtheorem{defn}{Definition}[section]
\newtheorem{lem}[defn]{Lemma}

\newtheorem{theorem}[defn]{Theorem}
\newtheorem{cor}[defn]{Corollary}
\newtheorem{rem}[defn]{Remark}

\newcommand{\qed}{\hfill $\square$}
\newcommand{\bs}{\boldsymbol}
\newenvironment{proof}{
      \noindent {\bf Proof }}{\qed
      \vspace{0.25\baselineskip}
}

\newcommand{\debproof}{\begin{proof}}
\newcommand{\finproof}{\end{proof}}

\definecolor{darkmagenta}{rgb}{0.5,0,0.5}
\definecolor{hotpink}{rgb}{1,0.2,1}
\definecolor{darkgreen}{rgb}{0,0.5,0}
\definecolor{darkblue}{rgb}{0,0,0.85}
\definecolor{darkred}{rgb}{0.8,0,0}
\definecolor{mellow}{rgb}{.847, 0.72, 0.525}

\newcommand{\bleu}[1]{\textcolor{black}{#1}}

\usepackage{pifont}

\begin{document}

\title{Doubling Algorithms for Stationary Distributions of Fluid Queues: A Probabilistic Interpretation}
\author{Nigel Bean\thanks{The University of Adelaide, School of Mathematical Sciences, SA 5005, Australia, and ARC Centre of Excellence for Mathematical and Statistical Frontiers \texttt{\{nigel.bean,giang.nguyen\}@adelaide.edu.au}} 
\and
Giang T. Nguyen\footnotemark[\value{footnote}]
\and 
Federico Poloni\footnote{Universit\`a di Pisa, Dipartimento di Informatica, Pisa, Italy, \texttt{federico.poloni@unipi.it}}
}
\date{{\today}}
\maketitle

\begin{abstract}
Fluid queues are mathematical models frequently used in stochastic modelling. Their stationary distributions involve a key matrix recording the conditional probabilities of returning to an initial level from above, often known in the literature as the matrix $\Psi$. Here, we present a probabilistic interpretation of the family of algorithms known as \emph{doubling}, which are currently the most effective algorithms for computing the return probability matrix $\Psi$. 

To this end, we first revisit the links described in \cite{ram99, soares02} between fluid queues and Quasi-Birth-Death processes; in particular, we give new probabilistic interpretations for these connections. We generalize this framework to give a probabilistic meaning for the initial step of doubling algorithms, and include also an interpretation for the iterative step of these algorithms. Our work is the first probabilistic interpretation available for doubling algorithms.    \\

\end{abstract}

\section{Introduction}

Stochastic fluid queues are two-dimensional Markov processes frequently used for modeling real-life applications. %
In a fluid queue $\{X_t, \varphi_t\}_{t \geq 0}$, the \emph{phase} $\varphi_t$ is a continuous-time Markov chain on a finite state space $\mathcal{S}$, and the \emph{level} $X_t \in (-\infty,\infty)$ varies linearly at rate $c_{\varphi_t}$. We consider the associated regulated process $\{\overline{X}_t, \varphi_t\}_{t \geq 0}$ that has a boundary at $0$: 
	\begin{align*}
		\overline{X}_t := X_t - \inf_{0 \leq s \leq t} X_s.
	\end{align*} 
	The joint stationary distribution of $\{\overline{X}_t, \varphi_t\}$ has been well-analyzed. Several authors independently derived this distribution using different approaches: time reversal~\cite{asmussen95}, the theory of generators of Markov processes and Wiener-Hopf factorization~\cite{roger94}, and partial differential equations~\cite{kk95}. More recently, Ramaswami~\cite{ram99} and da Silva Soares and Latouche~\cite{soares02} obtained new representations of the stationary distribution using matrix-analytic methods. 

The key component for obtaining the stationary distribution is the probability matrix $\Psi$, of which each entry $\Psi_{ij}$ is the probability of the fluid returning, from above, to the initial level $x$ in phase~$j$, after starting in phase~$i$ and avoiding all levels below $x$. This matrix $\Psi$ is also the minimal nonnegative solution to a \emph{nonsymmetric algebraic Riccati equation} (NARE) of the form 
\begin{align} 
	\label{eqn:nare} 
	B - A X - X D + X C X = 0. 
\end{align} 
 
These probabilistic and algebraic characterizations of $\Psi$ have led to considerable efforts in developing algorithms for computing the matrix efficiently. Asmussen~\cite{asmussen95} presented three iterative schemes, while Guo~\cite{guo01} analyzed fixed-point iterations and Newton's method. Following a different path, Ramaswami~\cite{ram99} and da Silva Soares and Latouche~\cite{soares02} proved that one can approximate fluid processes using Quasi-Birth-Death (QBD) processes, thus allowing quadratically convergent algorithms originally developed for QBDs---such as Logarithmic Reduction~\cite{lr93} and Cyclic Reduction~\cite{bm09}---to be used for solving for $\Psi$. Bean \emph{et al.}~\cite{bean05} proposed First-Exit and Last-Entrance and gave probabilistic interpretations for these two algorithms, as well as for Newton's method, one of Asmussen's iterative schemes, and Logarithmic Reduction applied to the QBD version of fluid processes. 

Probabilistic interpretations are useful as they give an intuitive explanation of how a numerical algorithm works, which in turn allows for shorter and more elegant proofs, as well as for improvements and generalizations of the algorithm. The combination of purely linear-algebraic manipulations and probabilistic understanding has paved ways for many significant theoretical developments, as already seen for QBDs~\cite{lr99, blm05} and fluid queues~\cite{ram99,soares02,bean05}.

In this paper, we focus on a family of algorithms for solving~\eqref{eqn:nare} known as \emph{doubling}, which include the structure-preserving doubling algorithm (SDA) \cite{glx05},  {SDA shrink-and-shift} (SDA-ss)~\cite{bmp}, {alternating-directional doubling algorithm} (ADDA)~\cite{wwl}, and component-wise accurate doubling algorithm~\cite{NguP15}. These algorithms are proven to be more computationally efficient than all the ones presented in~\cite{asmussen95, ram99, guo01, bean05, soares02}. We give these doubling algorithms a probabilistic interpretation, which to the best of our knowledge is the first one available.

In order to arrive at this probabilistic interpretation, we revisit the QBD approximations presented in \cite{ram99} and \cite{soares02} and give them a new, unified interpretation. Putting all three families of algorithms---those in \cite{ram99}, in \cite{soares02}, and doublings---under the same probabilistic light reveals clearly the two reasons why doubling algorithms are more efficient. First of all, they correspond to a QBD with a simpler transition structure, which allows one to reduce the number of matrix blocks that appear in each iteration. Second, they depend on two uniformization parameters instead of one, and these parameters can be chosen suitably to achieve a faster convergence rate.

Thus, the aim of this work is twofold. One is to understand doubling algorithms more thoroughly from a probabilistic point of view, with an eye to possible future generalizations. The other is to make these algorithms, which were developed purely from a linear algebra perspective, more accessible to probabilists and practitioners who work on stochastic fluid flows.

In Section~\ref{sec:sda}, we describe the doubling algorithms. In Section~\ref{sec:psi}, we present three algebraic representations for the return probability matrix $\Psi$ and their corresponding probabilistic interpretations. Based on the latter, we give in Section~\ref{sec:rei} three QBDs whose first downward return matrices $\GG$ contain the matrix $\Psi$ as its sub-block. Two of these QBDs are directly connected to the ones introduced in \cite{ram99, soares02}; the third is connected to the QBD in doubling algorithms. In Section~\ref{sec:int}, we give probabilistic meanings for the initial starting point and iterations of doubling algorithms, and discuss their convergence through the analysis of the matrices $\GG$ and $\RR$ of these QBDs. Some concluding remarks on the efficiency of these algorithms are in Section~\ref{sec:conclusion}. 

\section{Doubling algorithms}
	\label{sec:sda}
	
	Consider a regulated fluid model  $\{\overline{X}_t, \varphi_t\}_{t \geq 0}$, where $\varphi_t$ is a continuous-time Markov chain on a finite state space $\mathcal{S}$, $\overline{X}_t \in [0,\infty)$ and 
		\begin{align*} 
			\frac{\ud}{\ud t} \overline{X}_t = \left\{ \begin{array}{ll} 
				c_{\varphi_t} &  \mbox{ for } \overline{X}_t > 0, \\ 
				\vspace*{-0.3cm} \\
				\max\{c_{\varphi_t}, 0\} & \mbox{ for } \overline{X}_t = 0.
			 \end{array} \right.
		\end{align*} 

	Without loss of generality, we assume all fluid rates \bleu{$c_{\varphi_t}$} to be non-zero, \bleu{because} for any given model with zero rates we can censor out those zero-rate states without affecting the return probability matrix $\Psi$~\cite{soares05}.
		
	We define $C:= \mbox{diag}(c_i)_{i \in \mathcal{S}}$ to be the diagonal rate matrix for the level $\overline{X}_t$, and denote by $T$ the generator of the phase process $\varphi_t$. Let $\mathcal{S}_{+} := \{i \in \mathcal{S}: c_i > 0\}$, $\mathcal{S}_- := \{i \in \mathcal{S}: c_i < 0\}$, $n := |\mathcal{S}|$, $n_{+}: = |\mathcal{S}_{+}|$, and $n_{-}: = |\mathcal{S}_{-}|$.
	We partition the matrices $C$ and $T$ accordingly into sub-blocks
	\begin{align*} 
	C = \left[\begin{array}{cc} 
	C_{+} & \\ & C_{-} 
	\end{array} \right], \quad 
	T = \left[\begin{array}{cc}
		T_{++} & T_{+-} \\ 
		T_{-+} & T_{--} 
		\end{array}\right]. 
	\end{align*} 

It is well-established (see e.g.~\cite[Theorem~2]{roger94}) that the return probability matrix $\Psi$ is the minimal non-negative solution of the nonsymmetric algebraic Riccati equation
\begin{equation}
C_{+}^{-1}T_{+-} + \Psi|C_{-}|^{-1}T_{--} + C_{+}^{-1}T_{++}\Psi + \Psi|C_{-}|^{-1}T_{-+}\Psi = 0. \label{eqn:Psi}	
\end{equation}
In addition to $\Psi$, one can consider the return probability matrix $\widehat{\Psi}$ defined analogously on the process obtained by replacing $C$ with $-C$, i.e., reversing the sign of all rates (\emph{level-reversed} process). Similarly, this matrix $\widehat{\Psi}$ can be characterized as the minimal non-negative solution of the nonsymmetric algebraic Riccati equation
\begin{equation} 
|C_{-}|^{-1}T_{-+} + \widehat{\Psi}C_{+}^{-1}T_{++} + |C_{-}|^{-1}T_{--}\widehat{\Psi} + \widehat{\Psi} C_{+}^{-1}T_{+-} \widehat{\Psi} = 0. \label{eqn:Psihat} 
\end{equation}
We shall see that the doubling algorithms compute $\widehat{\Psi}$ alongside $\Psi$ without additional work.

In the remainder of this section, we describe \emph{operationally} the doubling algorithms, which are numerical methods to compute these two matrices $\Psi$ and $\widehat{\Psi}$.

\noindent\textbf{Initial starting point.} Let
\begin{align} 
\label{eqn:alphaopt}
\alpha_{\mathrm{opt}}:=\min_{i\in\mathcal{S}_-} \abs*{\frac{C_{ii}}{T_{ii}}} \quad \mbox{ and }  \quad \beta_{\mathrm{opt}}:=\min_{i\in\mathcal{S}_+} \abs*{\frac{C_{ii}}{T_{ii}}}.
\end{align}
Next, choose two real constants $\alpha,\beta$ with
\begin{equation} \label{eqn:alphabeta}
0 \leq \alpha \leq \alpha_{\mathrm{opt}}, \quad 0 \leq \beta \leq \beta_{\mathrm{opt}}, \quad \text{ $\alpha$, $\beta$ not both zero},	
\end{equation}
and set
	\begin{align}
Q & := \m{C_+-\alpha T_{++} & -\beta T_{+-}\\ -\alpha T_{-+} & \abs{C_-}-\beta T_{--}}, \\ 
R &:=\m{C_++\beta T_{++} & \alpha T_{+-}\\ \beta T_{-+} & \abs{C_-}+\alpha T_{--}}.
	\end{align}
and compute 
\begin{align}
	 \label{eqn:P0}
\begin{bmatrix}
    E & G\\H & F
\end{bmatrix}:=Q^{-1}R,
\end{align} 
with the same block sizes as above, i.e., $E\in\mathbb{R}^{n_+\times n_+}, F\in\mathbb{R}^{n_-\times n_-}, G\in\mathbb{R}^{n_+\times n_-}, H\in\mathbb{R}^{n_-\times n_+}$.

\noindent\textbf{Iterations.}
Define
\begin{subequations} 
\label{eqn:SDA}
\begin{align}
\widehat{E} & := E(I - G H)^{-1}E, \label{eqn:SDA_Ek} \\
\widehat{F} & := F(I - HG)^{-1}F, \label{eqn:SDA_Fk} \\
\widehat{G} & := G + E(I - GH)^{-1}GF,  \label{eqn:SDA_Gk} \\ 
\widehat{H} & := H + F(I - HG)^{-1}HE.  \label{eqn:SDA_Hk} 
\end{align}
\end{subequations}
The algorithm replaces at each step $(E,F,G,H)$ with the matrices $(\widehat{E},\widehat{F},\widehat{G},\widehat{H})$ computed according to~\eqref{eqn:SDA}, repeating until convergence. More formally, the algorithm can be implemented as follows.

\begin{algorithm}[H]
\KwIn{$T,C$ the transition probability matrix and rate matrix of a fluid queue (with rates $c_i\neq 0$ for all $i$); a threshold $\varepsilon$ (for instance, $\varepsilon=10^{-16}$)}
\KwOut{The return probability matrices $\Psi$ and $\widehat{\Psi}$}
Choose $\alpha,\beta$ according to~\eqref{eqn:alphabeta}\;
Compute initial $E,F,G,H$ according to~\eqref{eqn:P0}\;
\While{$\norm{E}\norm{F} > \varepsilon$}{
	Replace $(E,F,G,H)$ with $(\widehat{E},\widehat{F},\widehat{G},\widehat{H})$ computed according to~\eqref{eqn:SDA}\;
}
$\Psi = G$\;
$\widehat{\Psi} = H$\;
\caption{The pseudocode description of a doubling algorithm} \label{algo:doubling}
\end{algorithm}
We summarize here the convergence properties of this algorithm, which are described in~\cite{bmp,ChiCGMLX09,glx05,wwl}.
\begin{theorem} \label{thm:doublingconvergence}
In Algorithm~\ref{algo:doubling}, denoting by $E_k,F_k,G_k,H_k$ the values of the iterates at step $k$, one has:
\begin{enumerate}
	\item $0 \leq G_0 \leq G_1 \leq G_2 \leq \dots \leq \Psi$, and $\lim_{k\to\infty} G_k = \Psi$. 
	\item $0 \leq H_0 \leq H_1 \leq H_2 \leq \dots \leq \widehat{\Psi}$, and $\lim_{k\to\infty} H_k = \widehat{\Psi}$.
	\item If the fluid queue $\{X_t, \varphi_t\}$ is positive recurrent, then $\lim_{k\to\infty} E_k = 0$. If it is transient, then $\lim_{k\to\infty} F_k = 0$. If it is null recurrent, then both these equalities hold.
	\item The convergence rate for these limits is linear if $\{X_t, \varphi_t\}$ is null recurrent, and quadratic otherwise.
\end{enumerate}
\end{theorem}

We speak of \emph{doubling algorithms} in the plural because several variants have appeared in literature, differing only in the choice of the parameters $\alpha$ and $\beta$. The first one to appear, called SDA~\cite{glx05}, uses $\alpha=\beta=\min(\alpha_{\mathrm{opt}},\beta_{\mathrm{opt}})$. The variant called SDA-ss~\cite{bmp} uses $\alpha=0$, $\beta=\beta_{\mathrm{opt}}$. The variant called ADDA~\cite{wwl} uses $\alpha=\alpha_{\mathrm{opt}}$, $\beta=\beta_{\mathrm{opt}}$, and is the most computationally efficient of the three.

Our goal in the next sections is to give a probabilistic interpretation of this procedure.

\section{The Return Probability Matrix $\Psi$}
	\label{sec:psi} 
	We begin by generalizing the algebraic representation and probabilistic interpretation of the return probability matrix $\Psi$ presented in da Silva Soares and Latouche \cite{soares02}, in Section~\ref{sec:known}. New representations of $\Psi$ are presented in Section~\ref{sec:new}.
	
	\subsection{A known representation} 
		\label{sec:known}

	For simplicity, we assume that the level process $\{X_t\}$ has unit fluid rates only, that is, $c_i = \pm 1$ for all $i \in \mathcal{S}$. This assumption is not restrictive, since a fluid model with rates $\pm 1$ can always be obtained by a suitable time rescaling, and has the same return probability matrix $\Psi$. 

	Moreover, we assume that $X_0 = 0$ and $\varphi_0 \in \mathcal{S}_+$. Let $y$ be the time of the first transition into a phase in $\mathcal{S}_-$, which is also the level where the fluid process for the first time stops increasing and starts decreasing, by the assumption of unit rates. Hence $X_y = y$, $\varphi_y \in \mathcal{S}_-$.

	 Let $\tau(-x) := \inf \{t \geq y: X_t = y-x\}$ be the first time the process $\{X_t\}$ returns to level~$y-x$. 
	 Then, the time-changed phase process $\{\varphi_{\tau(-x)}\}_{x \geq 0}$ is also a Markov chain. Let $U$ be the $|\mathcal{S}_{-}| \times |\mathcal{S}_{-}|$ generator of the phase $\{\varphi_{\tau(-x)}\}$ of the \emph{downward record} process $\{X_{\tau(-x)}, \varphi_{\tau(-x)}\}_{x \geq 0}$. It is known (see e.g.~\cite[Equation~(13)]{soares02}) that 
		\begin{align}
			\label{eqn:U} 
			U = T_{--} +T_{-+} \Psi.  
		\end{align}  
		
 If we condition on $y$, a standard argument gives an expression for the return probability matrix $\Psi$:
\begin{align}
	\Psi = \int_0^\infty \exp(T_{++}y)T_{+-}\exp(Uy) \ud y. 
\end{align}

In~\cite{soares02}, the authors uniformized the first upward part of the path, from level~$0$ to level~$y$, and then the remaining part, from level~$y$ returning to level~$0$, both with the same rate. Here, we reproduce their approach in a more general fashion, in which there are two different uniformization rates, $\lambda$ and~$\mu$. 
	
Assume that $\varphi_0 \in \mathcal{S}_{+}$. Let $\mathcal{U}_{\lambda} := \{\mathcal{U}_{\lambda}(t)\}_{t \geq 0}$ (for `upwards') and $\mathcal{D}_{\mu} := \{\mathcal{D}_{\mu}(t)\}_{t \geq 0}$ (for `downwards') denote two Poisson processes with rates $\lambda$ and $\mu$, respectively. Consider a uniformization of the phase process $\{\varphi_t\}$ using $\mathcal{U}_{\lambda}$. Let $k +1$, $k \geq 0$, be the number of steps it takes until the uniformized process first switches from $\mathcal{S}_{+}$ to $\mathcal{S}_{-}$, at time and level $y$. 
Thus, there are $k$ events in the time interval $[0,y)$ and an event at time $y$. We use the process $\mathcal{D}_{\mu}$ for uniformizing with rate~$\mu$ the downward record process $\{\varphi_{\tau(-x)}\}$. Let $n \geq 0$ be the number of events of $\mathcal{D}_{\mu}$ in $[0,y]$, i.e., until the downward record process reaches level~$0$. 

Summing on $k$ and $n$, we obtain
\begin{align}
	\label{eqn:Psids}
	\Psi = \int_0^\infty \left[\sum_{k=0}^\infty \ue^{-\lambda y}\frac{(\lambda y)^k}{k!} P_{\lambda++}^k\right] \lambda P_{\lambda+-}\left[\sum_{n=0}^\infty \ue^{-\mu y}\frac{(\mu y)^n}{n!} V^n_{\mu}\right] \ud y,
\end{align}
with 
	%
	\begin{align} 
			\label{eqn:Plambda}
		P_\lambda & := I + \lambda^{-1}T, \\
			\label{eqn:Vmu} 
		V_{\mu} & := I + \mu^{-1}U.
	\end{align} 
Swapping the order of summation and integration gives
\begin{align} 
	\label{latouchesum}
	\Psi & = \sum_{k,n=0}^{\infty} \gamma_{k,n} P_{\lambda++}^kP_{\lambda+-}V_\mu^n, \\ 	
	\intertext{where} 
	\label{eqn:gammakn} 
\gamma_{k,n} 
& := \frac{(k+n)!}{k!n!}\frac{\lambda^{k+1}\mu^n}{(\lambda+\mu)^{k+n+1}}.
	\end{align} 
Figure~\ref{f:times} gives a sample path of $\{\overline{X}_t\}$, with $k = 3$ and $n = 4$.

\begin{figure}[h!]
    \centering
  		 \includegraphics[scale=0.7]{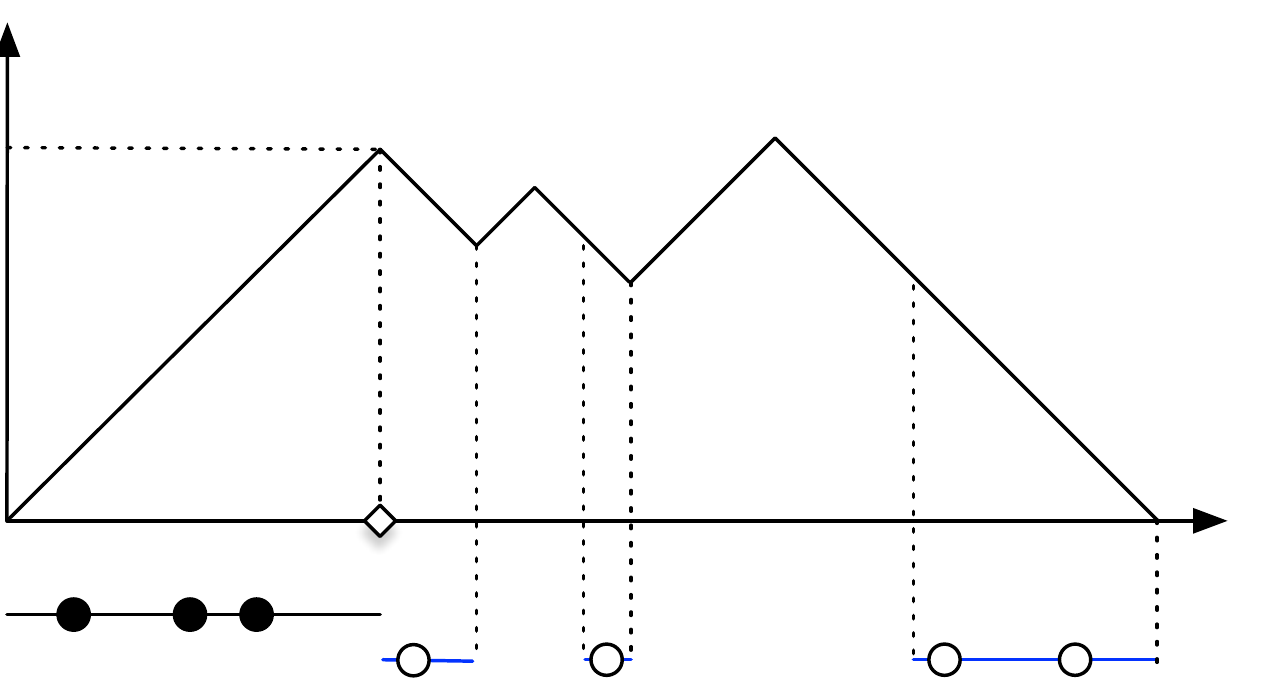}
		 \put(-265,115){\makebox(0,0)[l]{$y$}}
		 \put(-265,40){\makebox(0,0)[l]{$0$}} 
		  \put(-3,40){\makebox(0,0)[l]{$t$}} 
		  \put(-259,152){\makebox(0,0)[l]{$\overline{X}_t$}} 
     \caption{A sample path of $\{\overline{X}_t\}$, with $k = 3$ events $\bullet$ in the time interval $[0,y]$, an event $\diamondsuit$ at time $y$, and $n = 4$ events $\circ$ in the censored time it takes for the downward record process to decrease from $y$ to~$0$.}
    \label{f:times}
\end{figure}
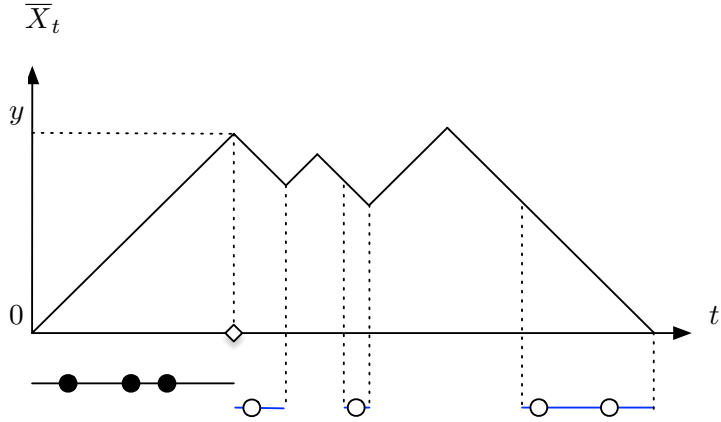

Da Silva Soares and Latouche~\cite{soares02} gave a probabilistic interpretation for the sum in~\eqref{latouchesum}. sample paths of the uniformized processes are decomposed into disjoint sets $\mathcal{A}_{k,n}$; the terms $P_{\lambda++}^kP_{\lambda+-}$ and $V_{\mu}^n$ account for the transitions in the two uniformized processes, while the coefficient $\gamma_{k,n}$ gives the probability to observe $k$ and $n$ events, respectively. To explain the formula~\eqref{eqn:gammakn}, consider the two Poisson process $\mathcal{U}_\lambda$ and $\mathcal{D}_\mu$: they are two independent Poisson processes which take place on two different intervals of time of length $y$: the first is the real time for the fluid process, on its way up to reaching the level $y$ for the first time; the second is the censored time of the downward record process, returning to level $0$ after starting at level $y$. Suppose instead that these two processes are superimposed onto the same time interval $[0,y]$, as in Figure~\ref{f:sum0-dots}.

\begin{figure}[ht]
    \centering
  		 \includegraphics[scale=0.7]{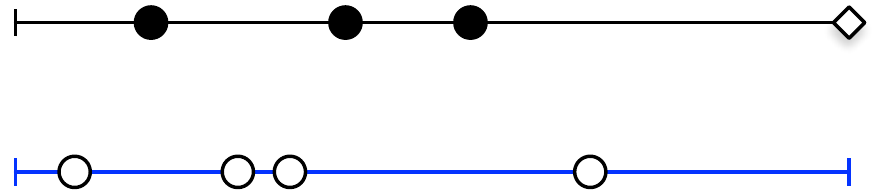}
		 \put(-180,48){\makebox(0,0)[l]{$u_0$}}	
		 \put(-149,48){\makebox(0,0)[l]{$u_1$}}		 
		 \put(-114,48){\makebox(0,0)[l]{$u_2$}} 
		 \put(-88,48){\makebox(0,0)[l]{$u_3$}}		
		 \put(-9,48){\makebox(0,0)[l]{$u_4$}}	 
		 \put(-188,36){\makebox(0,0)[l]{$0$}}		 
		 \put(2,36){\makebox(0,0)[l]{$y$}} 
		 \put(-188,4){\makebox(0,0)[l]{$0$}}
		 \put(2,4){\makebox(0,0)[l]{$y$}}	
		 %

    \caption{The processes $\mathcal{U}_\lambda$ and $\mathcal{D}_\mu$, superimposed on the same interval $[0,y]$.}
    \label{f:sum0-dots}
\end{figure}
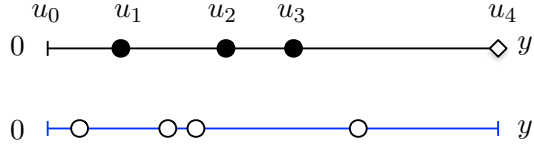
Then, $\gamma_{k,n}$ is the probability of observing $n$ events of the second process before the $k+1$st event of the first one. Since the probability of observing an event of the first process before one of the second is $\frac{\lambda}{\lambda+\mu}$, then $\gamma_{k,n}$ is the probability mass function of a negative binomial distribution with $k+1$ successes and probability $\frac{\lambda}{\lambda+\mu}$, which is exactly~\eqref{eqn:gammakn}.

\subsection{New representations} 
	\label{sec:new} 
	
Using probabilistic arguments, we consider three different representations for the matrix $\Psi$, each involving only one summation variable. Intuitively, Theorem~\ref{theo:sums} presents new ways of categorizing returning-to-zero sample paths of a fluid process, relying on two Poisson processes with rates $\lambda$ and $\mu$, respectively. They are the crucial foundation of our new interpretations for the links between QBDs and fluid queues, described in Section~\ref{sec:rei}. 

\begin{theorem}
	\label{theo:sums}
\begin{align} 
	\label{sum1}
\Psi & = \sum_{k=0}^{\infty} P_{\lambda++}^k P_{\lambda+-} \left(I- \lambda^{-1} U\right)^{-k-1} \\ 
	\label{sum2}
	& = \sum_{n=0}^{\infty} \left(I- \mu^{-1}T_{++}\right)^{-n-1}P_{\mu+-}V_\mu^n, \\
	& = \sum_{m=0}^{\infty} Q^m\left(I-\mu^{-1}T_{++}\right)^{-1} \left(P_{\lambda+-}W + P_{\mu+-}\right)W^m,  	\label{sum3}
	\end{align}
	where 
		\begin{align} 
			\label{eqn:W} 
			Q:= \left(I-\mu^{-1}T_{++}\right)^{-1}\left(I+\lambda^{-1}T_{++}\right), \quad
			W :=  \left(I + \mu^{-1}U\right)\left(I- \lambda^{-1}U\right)^{-1},
		\end{align} 
\end{theorem}
and $P_{\mu} := I + \mu^{-1}T$, analogously with $P_\lambda$.
\begin{rem} 
The spirit of the proof is as follows. We show that 
\begin{itemize} 
	\item[(i)] The first representation~\eqref{sum1} arises from uniformizing the fluid process on the way up to $y$ with rate $\lambda$. The Poisson process in the uniformization marks a sequence of exponential intervals of rate $\lambda$. We then observe the downward record process on the way down at these exponential epochs. 
	\item[(ii)] The second representation~\eqref{sum2} comes from the opposite: uniformizing the downward record process with rate $\mu$, and observing the initial upward path at these exponential intervals. 
	\item[(iii)] The third one~\eqref{sum3}, which leads to the QBD underlying the doubling algorithms, is a mixture of the previous two: alternating between a uniformization step and an observation, all the way up for the level process and then all the way down for the downward record process. The intervals at which we take actions---observing or uniformizing---are created by carefully superimposing the two Poisson processes. 
\end{itemize} 
\end{rem}

\begin{proof}

\textbf{First representation.} To show Equality~\eqref{sum1}, we denote by $\{u_i\}_{i = 1, \ldots, k} $ the sequence of time epochs associated with the $k$ Poisson events of $\mathcal{U}_{\lambda}$ in $(0,y)$, where $0 < u_1 < \cdots < u_k < y$, and define $u_0 := 0$ and $u_{k + 1} := y$. Next, let $m_i$, $i = 1, \ldots, k+1$, be the number of events of $\mathcal{D}_{\mu}$ during $[u_{i - 1}, u_i)$. Figure~\ref{f:sum1-dots} gives a sample path with $k = 3$.

\begin{figure}[ht]
    \centering
  		 \includegraphics[scale=0.7]{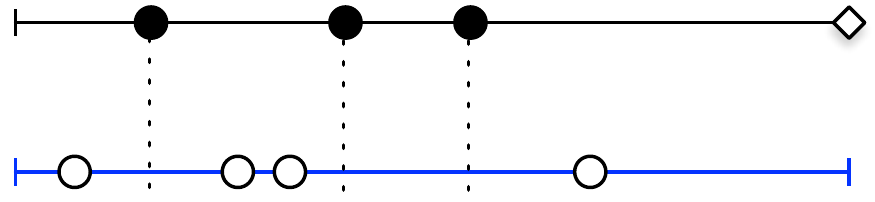}
		 \put(-180,48){\makebox(0,0)[l]{$u_0$}}	
		 \put(-149,48){\makebox(0,0)[l]{$u_1$}}		 
		 \put(-114,48){\makebox(0,0)[l]{$u_2$}} 
		 \put(-88,48){\makebox(0,0)[l]{$u_3$}}		
		 \put(-9,48){\makebox(0,0)[l]{$u_4$}}	 
		 \put(-188,36){\makebox(0,0)[l]{$0$}}		 
		 \put(2,36){\makebox(0,0)[l]{$y$}} 
		 \put(-188,4){\makebox(0,0)[l]{$0$}}
		 \put(2,4){\makebox(0,0)[l]{$y$}}		
    \caption{A sample path with $k = 3$ uniformization events $\bullet$ of $\mathcal{U}_{\lambda}$ in the time interval $[0,y)$, $m_1 = 1$, $m_2 = 2,$ $m_3 = 0$, and $m_4 = 1$.}
    \label{f:sum1-dots}
\end{figure}
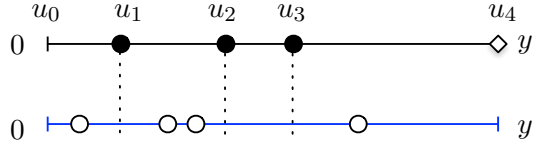

During the first ascent, at each event $u_i$, $i = 1, \ldots, k$, there is a uniformization step $P_{\lambda++}$, which leads to the term $P_{\lambda++}^{k}$ in the sum~\eqref{sum1}. At time $y$, there is a uniformization step associated with a transition from $\mathcal{S}_+$ to $\mathcal{S}_-$, which gives the term $P_{\lambda+-}$.

Recall that $V_\mu := 1 + \mu^{-1} U$ is the probability transition matrix of the uniformized downward record phase process $\left\{\varphi_{\tau(-x)}\right\}$, and transitions according to $V_{\mu}$ happen at the events of $\mathcal{D}_\mu$. By the same argument used to justify~\eqref{eqn:gammakn}, the probability that there are $m_i$ events of $\mathcal{D}_\mu$ in $[u_{i-1},u_i)$ is $\mu^{m_i}\lambda / (\lambda+\mu)^{m_i+1}$.

Let $P_{i}$ be the matrix with elements $
		[P_i]_{jk} := \mathds{P}[\bleu{\varphi_{\tau( - u_{i})}} = k | \bleu{\varphi_{\tau( - u_{i - 1})}} = j]$ for $j, k \in \mathcal{S}_{-}.$ Then, 
	\begin{align}
		\label{eqn:Pd}
	      P_i & = \sum_{m_i = 0}^{\infty} \frac{\mu^{m_i} \lambda }{(\lambda + \mu)^{m_i + 1}}V_{\mu}^{m_i} 
	       = \frac{\lambda}{\lambda + \mu} \left(I - \frac{\mu}{\lambda + \mu} V_{\mu} \right)^{-1}  
	       = \left(I - \frac{1}{\lambda}U\right)^{-1}.
	\end{align} 
	
	Since there are \bleu{$k + 1$} intervals $[u_{i - 1}, u_{i})$, $i = 1, \ldots, k + 1$, this explains the term $(I - \lambda^{-1}U)^{\bleu{-k - 1}}$ in, and completes the proof of, \eqref{sum1}.
	
	Note that as $\left(I - {\lambda}^{-1}U\right)^{-1} = \int_0^{\infty} \lambda \ue^{-\lambda x} \ue^{Ux} \ud x$, the matrix $P_i$ contains the conditional probabilities of the downward record process observed at exponential intervals with rate~$\lambda$. Thus, the term $(I - \lambda^{-1}U)^{-k - 1}$ can also be seen directly as arising from observing the downward record process at the $(k + 1)$ exponentially distributed intervals marked by the epochs $u_i$, and we could have skipped the construction of $\mathcal{D}_{\mu}$ altogether. However, we need this process to explain the following representations, and both $\mathcal{U}_{\lambda}$ and $\mathcal{D}_{\mu}$ to explain the third, so we maintain both processes throughout the proof.
	
	\textbf{Second Representation.}  Equality~\eqref{sum2} comes from grouping the sample paths in a different manner to that in the first representation: according to the number of events, $n$, of the process $\mathcal{D}_{\mu}$ in $[0,y]$. This provides the $n$ uniformizing steps of the downward record process, represented by the term $V_{\mu}^n$ in the sum \bleu{in}~\eqref{sum2}. 
	
	Let $\{d_i\}_{i = 1, \ldots, n}$ be the associated sequence of time epochs, and define $d_0 := 0$. \bleu{Figure~\ref{f:sum2-dots} represents a sample path with $n = 4$.}
	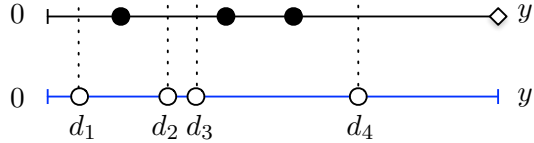
\begin{figure}[h!]
    \centering
  		 \includegraphics[scale=0.7]{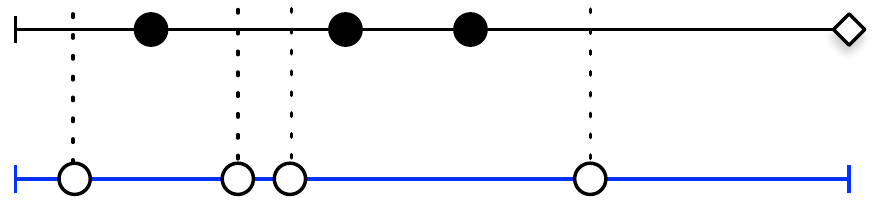}
		 \put(-166,-7){\makebox(0,0)[l]{$d_1$}}	
		 \put(-135,-7){\makebox(0,0)[l]{$d_2$}}			 
		 \put(-122,-7){\makebox(0,0)[l]{$d_3$}}
            	 \put(-62,-7){\makebox(0,0)[l]{$d_4$}}	
		 \put(-188,36){\makebox(0,0)[l]{$0$}}		 
		 \put(2,36){\makebox(0,0)[l]{$y$}} 
		 \put(-188,4){\makebox(0,0)[l]{$0$}}
		 \put(2,4){\makebox(0,0)[l]{$y$}}		 
    \caption{A sample path with $n = 4$ uniformization events $\circ$ of $\mathcal{D}_{\mu}$.}
    \label{f:sum2-dots}
\end{figure}

By arguments analogous to the ones used for the first representation, \bleu{during} the initial upward journey from level~$0$ to level~$y$ we observe the phase process $\{\varphi_t\}$ at the end of each exponential interval $[d_{i}, d_{i + 1})$, for $i = 0, \ldots, \bleu{n - 1}$. \bleu{This explains  the term $(1 - \mu^{-1}T_{++})^{-n}$ in \eqref{sum2}.} 

The last interval, $[d_k, y]$, is not an exponential interval with parameter~$\mu$,  \bleu{and we proceed through a different argument for the remaining unexplained term $\left(1 - \mu^{-1}T_{++}\right)^{-1}\mu^{-1}T_{+-}$. Note that in this interval there can be any number of events of type $\mathcal{U}_{\lambda}$, none of type $\mathcal{D}_{\mu}$, and precisely one event at time $y$ to transition from $\mathcal{S}_{+}$ to $\mathcal{S}_{-}$. Thus, the probability transition matrix for the state at time $y$, given the state at time $d_k$, is given by} 
	\begin{align} 
		\sum_{s = 0}^{\infty} \left(\frac{\lambda}{\lambda + \mu}\right)^{s + 1}P_{\lambda ++}^{s} P_{\lambda +-} &= \frac{\lambda}{\lambda+\mu}\left(I- \frac{\lambda}{\lambda+\mu} P_{\lambda++}\right)^{-1} \lambda^{-1} T_{+-} \nonumber\\
		&= \left((\lambda+\mu)I - \lambda P_{\lambda++}\right)^{-1} T_{+-} \nonumber\\
		&= \left(I - \mu^{-1}T_{++}\right)^{-1}\mu^{-1}T_{+-} \label{endcount}
	\end{align}
	
\textbf{Third Representation.} To arrive at~\eqref{sum3}, we superimpose $\mathcal{U}_{\lambda}$ and $\mathcal{D}_{\mu}$ and define a new sequence \bleu{on $[0,y)$} based on \bleu{the previously defined sequences} $\{u_i\}_{\bleu{i \geq 0}}$ and $\{d_i\}_{\bleu{i \geq 0}}$ as follows 
	\begin{align} 
		\label{eqn:c0} 
		c_0 & = 0, \\ 
		\label{eqn:codd} 
		c_{2i + 1} & = \min\{d_i: d_i > c_{2i}\}, \\
		\label{eqn:ceven} 		
		c_{2i + 2} & = \min\{u_i: u_i > c_{2i+1}\} \quad \mbox{for } i \geq 0.
	\end{align} 

	 For $c_{2i + 1} < y$, each interval $[c_{2i}, c_{2i + 1}]$ contains an arbitrary number of events of $\mathcal{U}_\lambda$ and exactly one event of $\mathcal{D}_{\mu}$, which lies at its right endpoint, and thus the interval has an exponential length with parameter $\mu$. On the other hand, for $c_{2i + 2} < y$, each interval $[c_{2i +1}, c_{2i + 2}]$ contains an arbitrary number of events of $\mathcal{D}_\mu$ and exactly one event of $\mathcal{U}_{\lambda}$, which lies at its right endpoint, and thus its length is exponentially distributed with parameter $\lambda$. 
	
	Let $N$ be the number of events $c_j$ in $[0,y)$. There are two possibilities: $N$ is even \bleu{(we include here the case $N = 0$)}, and $N$ is odd. We separate the summation in~\eqref{sum3} into two sums:
	\begin{align} 
		\label{eqn:sum3sep}
		\Psi & = \sum_{m=0}^{\infty} Q^m\left(I-\mu^{-1}T_{++}\right)^{-1} P_{\lambda+-}W^{m + 1} \nonumber \\
		& \;\; + 
		\sum_{m=0}^{\infty} Q^m\left(I-\mu^{-1}T_{++}\right)^{-1} P_{\mu+-}W^m I, 
	\end{align} 
	and show that the first sum comes from odd values of $N$, and the second from even ones. We have added an identity matrix as the final factor of the second sum; the reason will be explained later. 
	
	\bleu{First, consider the case when $N$ is odd.} Suppose that $N = 2m + 1$ for some integer $m \geq 0$. Figure~\ref{f:dots-odd} gives a sample path for $N = 5$. 

	\begin{figure}[h!]
    		\centering
  		 \includegraphics[scale=0.7]{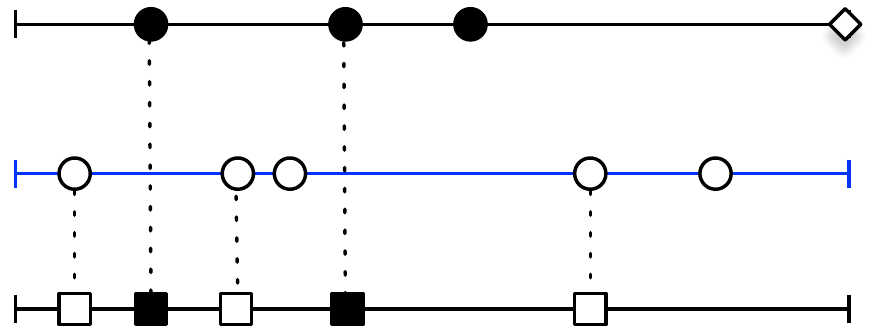}
		 \put(1,4){\makebox(0,0)[l]{$y$}}
		 \put(-186,4){\makebox(0,0)[l]{$0$}}
		 \put(1,34){\makebox(0,0)[l]{$y$}}
		 \put(-186,34){\makebox(0,0)[l]{$0$}}
		 \put(1,64){\makebox(0,0)[l]{$y$}}
		 \put(-186,64){\makebox(0,0)[l]{$0$}}
		 \put(-165,45){\makebox(0,0)[l]{$d_1$}}
		 \put(-165,-5){\makebox(0,0)[l]{$c_1$}}
		 \put(-151,74){\makebox(0,0)[l]{$u_1$}}		 		
		 \put(-151,-5){\makebox(0,0)[l]{$c_2$}}	
		 \put(-133,45){\makebox(0,0)[l]{$d_2$}}		   
		 \put(-133,-5){\makebox(0,0)[l]{$c_3$}}
		 \put(-112,74){\makebox(0,0)[l]{$u_2$}}		 
		 \put(-86,74){\makebox(0,0)[l]{$u_3$}}		 	
		 \put(-111,-5){\makebox(0,0)[l]{$c_4$}}	
		 \put(-62,45){\makebox(0,0)[l]{$d_4$}}			 
		 \put(-62,-5){\makebox(0,0)[l]{$c_5$}}	 
		 \put(-172,12){\makebox(0,0)[l]{$\mu$}}	
		 \put(-158,14){\makebox(0,0)[l]{$\lambda$}}	
		 \put(-140,12){\makebox(0,0)[l]{$\mu$}}			 		\put(-120,14){\makebox(0,0)[l]{$\lambda$}}	
		 \put(-84,12){\makebox(0,0)[l]{$\mu$}}					\put(-34,14){\makebox(0,0)[l]{$\lambda$}}			 
    \caption{A sample path of $\mathcal{U}_{\lambda}$ and $\mathcal{D}_{\mu}$, where $N = 5$.}
    \label{f:dots-odd}
\end{figure}
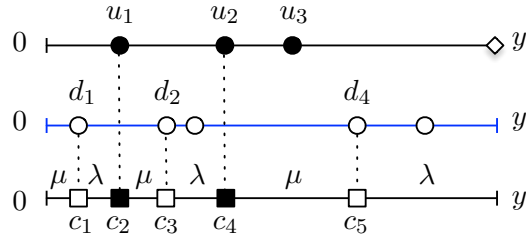	 
	
	Let $c_{2m + 2} := y$, then the sequence $\{c_i\}_{i = 0, \ldots, 2m + 2}$ mark $(2m + 2)$ exponentially distributed intervals, with rates alternating between $\mu$ and $\lambda$, starting with~$\mu$ and ending with $\lambda$. The last interval $[c_{N}, y]$ is exponential with rate $\lambda$ because it contains exactly one event of $\mathcal{U}_{\lambda}$, at time~$y$.  
During the upward path to $y$, for the first $(2m + 1)$ intervals, we alternate between
	\begin{itemize} 
	\item observing the fluid process at the end of an interval exponentially distributed with rate $\mu$---represented by $(I - \mu^{-1}T_{++})^{-1}$, and
	\item doing a uniformization step at the end of an interval exponentially distributed with rate $\lambda$---represented by $(I + \lambda^{-1}T_{++}) = P_{\lambda++}$. 
	\end{itemize} 
	This alternation leads to the product $Q^m\left(I-\mu^{-1}T_{++}\right)^{-1}$ in the sum~\eqref{sum3}. Then, at $y$ by definition there is a uniformization step with a switch from the upward direction to downward for the fluid $\{\overline{X}_t\}$, which gives rise to the term $P_{\lambda+-}$.
	
	On the other hand, during the time it takes for the downward process to decrease from $y$ to $0$, we alternate between observing the process at the end of an exponential interval $[c_i, c_{i + 1})$, represented by $(I - \lambda^{-1}U)^{-1}$, and doing a uniformization step with $(I + \mu^{-1}U)=V_{\mu}$. This gives the factor $W^{m + 1}$. Thus, the case of $N$ being odd explains the first summation in~\eqref{eqn:sum3sep}. 
	
	Now, suppose $N = 2m$ for some integer $m \geq 0$. The sequence $\{c_i\}_{i =0,\dots,2m}$ mark \bleu{$2m$} exponential intervals, with rates alternating between $\mu$ and $\lambda$. 


	Figure~\ref{f:dots-even} gives a sample path for $N = 4$ and $r = 2$.  

	\begin{figure}[h!]
    		\centering
  		 \includegraphics[scale=0.7]{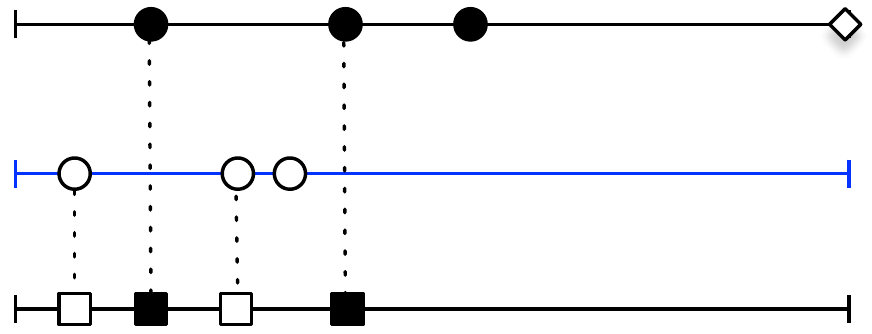}
	 \put(1,4){\makebox(0,0)[l]{$y$}}
		 \put(-186,4){\makebox(0,0)[l]{$0$}}
		 \put(1,34){\makebox(0,0)[l]{$y$}}
		 \put(-186,34){\makebox(0,0)[l]{$0$}}
		 \put(1,64){\makebox(0,0)[l]{$y$}}
		 \put(-186,64){\makebox(0,0)[l]{$0$}}
		 \put(-165,45){\makebox(0,0)[l]{$d_1$}}
		 \put(-165,-5){\makebox(0,0)[l]{$c_1$}}
		 \put(-151,74){\makebox(0,0)[l]{$u_1$}}		 		
		 \put(-151,-5){\makebox(0,0)[l]{$c_2$}}	
		 \put(-133,45){\makebox(0,0)[l]{$d_2$}}		   
		 \put(-133,-5){\makebox(0,0)[l]{$c_3$}}
		 \put(-112,74){\makebox(0,0)[l]{$u_2$}}		 
		 \put(-86,74){\makebox(0,0)[l]{$u_3$}}		 	
		 \put(-111,-5){\makebox(0,0)[l]{$c_4$}}		 
		 \put(-172,12){\makebox(0,0)[l]{$\mu$}}	
		 \put(-158,14){\makebox(0,0)[l]{$\lambda$}}	
		 \put(-140,12){\makebox(0,0)[l]{$\mu$}}			 		\put(-120,14){\makebox(0,0)[l]{$\lambda$}}		 
    \caption{A sample path of $\mathcal{U}_{\lambda}$ and $\mathcal{D}_{\mu}$, where $N = 4$ and $r = 2$.}
    \label{f:dots-even}
\end{figure}
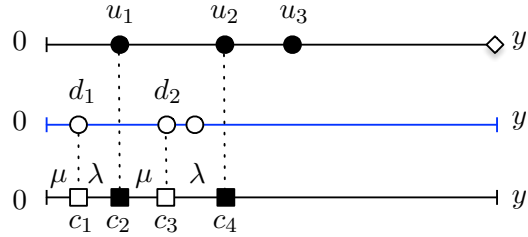

	Similar to when $N$ is odd, \bleu{for the upward part and until $c_{2m}$,} we alternate between observing the fluid process at the end of an exponential interval and doing a uniformization step, contributing a factor $Q^m = \left((I - \mu^{-1}T_{++})^{-1}(I + \lambda^{-1}T_{++})\right)^{m}$. The last interval of the upward path, $[c_{2m}, y]$, contributes a factor $(I - \mu^{-1}T_{++})^{-1}P_{\mu + -}$, by the same computation as in~\eqref{endcount}.

	Then, for the downward record process we also alternate between observing and uniformizing for the first $2m$ intervals, which leads to $W^{m}$. We need an additional factor to account for phase changes in the downward processes between time $c_{2m}$ and time $y$: however, there are no further events of $\mathcal{D}_\mu$ after the one at time $c_{2m}$; hence the downward uniformized process makes no transitions in $[c_{2m},y]$, and the missing factor is simply the identity matrix.
	\end{proof}

\section{Connections between QBDs and Fluid Queues}
	\label{sec:rei}

Recall that a discrete-time Quasi-Birth-Death process $\{Y_n, \kappa_n\}_{n \geq 0}$ is a two-dimensional Markov process, where the phase $\{\kappa_n\}$ is a Markov chain on a finite state space $\mathcal{M}$, the level $Y_n \in \mathds{Z}$ evolves between adjacent levels according to three transition probabilities, $L_-, L_0, L_+$, as follows:
	\begin{align*} 
		\mathds{P}[Y_n = k - 1, \kappa_{n} = j \mid Y_{n - 1} = k, \kappa_{n - 1} = i] & = [L_-]_{ij}, \\
		\mathds{P}[Y_n = k + 1, \kappa_{n} = j \mid Y_{n - 1} = k, \kappa_{n - 1} = i] & = [L_+]_{ij}, \\
		\mathds{P}[Y_n = k, \kappa_{n} = j \mid Y_{n - 1} = k, \kappa_{n - 1} = i] & = [L_{0}]_{ij}.
	\end{align*}  

Let $\theta(s) := \inf\{n > 0: Y_n = s\}$ be the first time the QBD hits level~$s$. A key matrix of interest for QBDs is the so-called \emph{first downward return} matrix $\GG$, whose elements 
	\begin{align*} 
		\GG_{ij} := \mathds{P}\left[\kappa_{\theta(k - 1)} = j,\, \theta(k - 1) < \infty \mid Y_{0} = k,\, \kappa_0 = i\right]
	\end{align*} 
	are the probabilities of returning to a level below the initial starting level $k$. For convenience, we refer to this matrix as the $\GG$-matrix.  
	
	\bleu{For all QBDs introduced henceforth,} the phase process takes values in $\mathcal{S} = \mathcal{S}_{+} \cup \mathcal{S}_-$, the same state space as that of the phase process $\varphi_t$ of the fluid model $\{\overline{X}_t, \varphi_t\}$, and we partition the transition matrices between QBD levels into sub-blocks corresponding to states in $\mathcal{S}_+$ and $\mathcal{S}_-$.
	
	\begin{defn} 
	In this paper, two Quasi-Birth-Death processes are said to be \emph{$\GG$-equivalent} if they share the same $\GG$-matrix.
	\end{defn}
	
	First, we construct in Section~\ref{subsec:newQBDs} three Quasi-Birth-Death processes whose $\GG$-matrices contain the matrix $\Psi$ as one of its \bleu{sub-}blocks, based directly on the new \bleu{representations of $\Psi$ in Section~\ref{sec:new}}. Then, we show in Section~\ref{subsec:newerQBDs} that the aforementioned QBDs are respectively $\GG$-equivalent to three other processes. 
	
	In particular, the first two QBDs are $\GG$-equivalent to the QBD constructed in Ramaswami~\cite[Theorem 16]{AhnR03} and in da Silva Soares and Latouche \cite[Eqn. (17)]{soares02}, respectively. The third QBD---which will be easily seen to reduce to the first QBD, when $\mu \rightarrow \infty$, or to the second QBD, when $\lambda \rightarrow \infty$, and is thus the most general form of the three QBDs---is the underlying QBD of the doubling algorithms.

\subsection{QBDs constructed from new representations of $\Psi$} 
	\label{subsec:newQBDs}

\begin{lem} 
	\label{lem:firstrep} 
	Consider a Quasi-Birth-Death process $\{Y^{A}_n, \kappa^{A}_n\}_{n \geq 0}$ with the following transition probability matrices: 
		\begin{align*}
			A_{-1} := \left[\begin{array}{cc}
			0 & 0 \\
			0 & \left(I - \lambda^{-1}U\right)^{-1} 
			\end{array} \right], \; 
			A_0 := \left[\begin{array}{cc} 
					0 & P_{\lambda +-} \\
					0 & 0
				  \end{array}\right], \; 
			A_1 := \left[\begin{array}{cc} 
				P_{\lambda ++} & 0 \\ 
				0 & 0
			\end{array} 
			\right].
		\end{align*} 
		Then, the $\GG$-matrix of this QBD is given by
\begin{align} 
	\label{eqn:GA-Ram}
\GG_{A} & =\begin{bmatrix}
  0 & \Psi\\
  0 & \left(I- \lambda^{-1}U\right)^{-1}
\end{bmatrix}, 
\end{align}
where $\Psi$ is the minimal nonnegative solution to $\eqref{eqn:Psi}$, and the generator $U$ is given by \eqref{eqn:U}. 
\end{lem}

\begin{proof}
Note that this QBD has a simple dynamic whose paths mimic directly the sum in \eqref{sum1}: starting from a phase in $ \mathcal{S}_{+}$ and level $Y^A_0 = 1$, the process 
			\begin{itemize} 
				\item stays in $\mathcal{S}_{+}$ for a certain number $k \geq 0$ of steps, each time going up one level and making a transition with $P_{\lambda ++}$, which means $Y^A_k = k+1$;
				\item then moves to $\mathcal{S}_{-}$ at the same level with transition matrix $P_{\lambda +-}$, hence $Y^{A}_{k+1} = k+1$;
				\item once having reached $\mathcal{S}_{-}$, the process cannot leave this set of states, and can only go down one level at each step, with stochastic transition matrix $\left(I - \lambda^{-1}U\right)^{-1}$, until it reaches level $0$ after $k+1$ further transitions.
			\end{itemize}
			Thus, the first passage probability matrix to a lower level, starting from a state in $\mathcal{S}_{+}$, is $\sum_{k \geq 0} P_{\lambda ++}^{k} P_{+-} \left(I - \lambda^{-1}U\right)^{-k - 1}$, which equals $\Psi$ by \eqref{sum1}. 
			
			Starting from a state in $\mathcal{S}_{-}$, the probability of going to a level lower than the initial one is $\left(I - \lambda^{-1}U\right)^{-1}$, which completes the proof. 
\end{proof}

\begin{lem} 
	\label{lem:secondrep} 
	Consider a Quasi-Birth-Death process $\{Y_n^{B}, \kappa_n^{B}\}_{n \geq 0}$ with the transition matrices 
		\begin{align*}
			B_{-1} & := \left[\begin{array}{cc} 
			0 & \left(I - \mu^{-1}T_{++}\right)^{-1} P_{\mu +-} \\
			0 & V_{\mu} 
			\end{array} \right], \; B_{0} := 0, 
			\\
			B_{1} & := \left[\begin{array}{cc} 
				\left(I - \mu^{-1}T_{++}\right)^{-1} & 0 \\
				0 & 0  
			\end{array} \right]. 
		\end{align*} 
Then, the $\GG$-matrix of this QBD is given by
\begin{align} 
	\label{eqn:GB-Guy}
\GG_B = \begin{bmatrix}
  0 & \Psi\\
  0 & V_{\mu}
\end{bmatrix},
\end{align}
where $\Psi$ is the minimal nonnegative solution to \eqref{eqn:Psi}, and the transition probability matrix $V_{\mu}$ is given by 
\eqref{eqn:Vmu}.
\end{lem}

\begin{proof} Similar to the QBD in Lemma~\ref{lem:firstrep}, the QBD here has a simple dynamic, whose paths mimic directly the sum in~\eqref{sum2}. Starting with $\kappa_0^{B} \in \mathcal{S}_{+}$ and $Y_{0}^{B} = 1$, the process 
			\begin{itemize} 
				\item stays in a state in $\mathcal{S}_{+}$ for a certain number $n \geq 0$ of steps, each with probability matrix $(I - \mu^{-1}T_{++})^{-1}$, which means $Y_{n}^{B} = n + 1$;
				\item then moves to a state in $\mathcal{S}_{-}$ and goes down a level (and hence $Y_{n + 1}^{B} = n$), according to $\left(I - \mu^{-1}T_{++}\right)^{-1} P_{\mu +-}$; 
				\item from now on, the process stays in $\mathcal{S}_{-}$, going down one level with stochastic transition matrix $V_{\mu}$ at each time step and reaching level $0$ after $n$ further transitions. 
			\end{itemize} 
			Thus, the first passage probability matrix to a lower level, starting from a state in $\mathcal{S}_{+}$, is $\sum_{n \geq 0} \left(I - \mu^{-1}T_{++}\right)^{-n - 1}P_{\mu +-} V_{\mu}^n$, which is $\Psi$ by \eqref{sum2}. 
			
			Starting from a state in $\mathcal{S}_{-}$, the probability of going to a level lower than the initial one is $V_{\mu}$, which completes the proof. 
\end{proof} 

\begin{lem} 
	\label{lem:thirdrep} 
	Consider a Quasi-Birth-Death process $\{Y_n^{C}, \kappa_n^{C}\}_{n \geq 0}$ with the transition matrices 
	 \begin{align*}
	C_{-1} & := \left[\begin{array}{cc} 
	0 & {\left(I - \mu^{-1} T_{++}\right)^{-1}P_{\mu {+-}}} \\ 
	0 & W \end{array}\right], \; 
	C_{0} := \left[\begin{array}{cc} 
	0 & (I - \mu^{-1}T_{++})^{-1}P_{\lambda +-} \\ 
	0 & 0 
	\end{array} \right], \\
	 C_{1} & := \left[\begin{array}{cc} 
	 	\left(I - \mu^{-1}T_{++}\right)^{-1}P_{\lambda ++} & 0 \\
		0 & 0 
	 \end{array} \right].
	\end{align*} 

	Then, its $\GG$-matrix is given by
\begin{align} \label{eqn:GC}
\GG_C = \begin{bmatrix}
  0 & \Psi\\
  0 & W
\end{bmatrix},
\end{align}
where $W$ is defined in $\eqref{eqn:W}$. 
\end{lem} 

\begin{proof} 
	The QBD $\{Y_n^{C}, \kappa_n^{C}\}_{n \geq 0}$ has sample paths that mimic the dynamics of the sum \eqref{sum3}. Starting from $\mathcal{S}_{+}$ and level $1$, the process 
	\begin{itemize}
		\item stays in a state in $\mathcal{S}_+$ for a certain number $m\geq 0$ of steps, each with probability matrix $\left(I - \mu^{-1}T_{++}\right)^{-1}P_{\lambda ++}$, which means that $Y_m^C = m+1$;
	\item then moves to the set $\mathcal{S}_{-}$, either moving down a level, with transition matrix $\left(I - \mu^{-1}T_{++}\right)^{-1}P_{\mu +-}$, or staying in the same level with $\left(I - \mu^{-1}T_{++}\right)^{-1}P_{\lambda +-}$. In the former case, $Y_{m+1}^C=m$, while in the latter  $Y_{m+1}^C = m + 1$; 
	\item having reached $\mathcal{S}_{-}$, the process stays there and reaches level $0$ after either $m$ or $m + 1$ transitions, depending on the choice made in the previous point. 
	\end{itemize} 
	
	Summing up the contributions of these paths we get the expression for $\Psi$ in \eqref{sum3}. 
	
	Thus, it is clear that the $\GG$-matrix for the QBD is given by $\GG_C$. 
\end{proof} 

\subsection{Connections to existing QBDs and doubling algorithms} 
	\label{subsec:newerQBDs}


\bleu{Ahn and Ramaswami \cite[\bleu{Theorem 16}]{AhnR03} constructed a Quasi-Birth-Death process $\{Y_n^{\Delta}, \kappa_n^{\Delta}\}_{n \geq 0}$ with the following transition probability matrices: 
\begin{align} \label{eqn:qbdram} 
\Delta_{-1} & := \frac12\begin{bmatrix}
  0 & 0\\
  0 & I
\end{bmatrix}, \;
\Delta_0 := \begin{bmatrix}
 0 &  P_{\lambda+-}\\
  0 & \frac12 P_{\lambda--}
\end{bmatrix}, \;
\Delta_1 := \begin{bmatrix}
  P_{\lambda++} & 0\\
  \frac12 P_{\lambda-+} & 0
\end{bmatrix},
\end{align}
and showed that the $\GG$-matrix of this QBD has $\Psi$ as its top-right sub-block.}
\begin{theorem}
	\label{theo:qbd-1}
Consider a Quasi-Birth-Death process $\{Y^{A'}_n, \kappa^{A'}_n\}_{n \geq 0}$ with the following transition probability matrices: 
\begin{align} 
		\label{qbdrammod}
A'_{-1} & := \frac12\begin{bmatrix}
  0 & 0\\
  0 & I
\end{bmatrix}, \;
A'_0 := \frac12\begin{bmatrix}
  I & P_{\lambda+-}\\
  0 & P_{\lambda--}
\end{bmatrix}, \;
A'_1 := \frac12\begin{bmatrix}
  P_{\lambda++} & 0\\
  P_{\lambda-+} & 0
\end{bmatrix}.
\end{align}

\bleu{The process $\{Y^{A'}_n, \kappa^{A'}_n\}_{n \geq 0}$ is $\GG$-equivalent to the QBD $\{Y_n^{\Delta}, \kappa_n^{\Delta}\}_{n \geq 0}$  and to the QBD $\{Y_n^{A}, \kappa_n^{A}\}$ introduced in Lemma~\ref{lem:firstrep}.}
\end{theorem}

\begin{proof}
\bleu{The first part is straightforward, the only difference between the dynamics of the two QBDs is in the transitions from a state in $\mathcal{S}_{+}$. For each step, with equal probabilities the process $\{Y_{n}^{A'}, \kappa_n^{A'}\}$ either 
	\begin{itemize} 
		\item follows the same dynamics as those in the QBD $\{Y_n^{\Delta}, \kappa_n^{\Delta}\}$, \emph{or} 
		\item makes a transition to the same state and the same level, that is, stays exactly where it is. 
	\end{itemize} 
	While this difference means the expected time for $\{Y_{n}^{A'}, \kappa_n^{A'}\}$ to descend one level is slower than that of $\{Y_{n}^{\Delta}, \kappa_n^{\Delta}\}$, the two processes have the same conditional probabilities of doing so, and thus the same $\GG$-matrix.}  
	
	We now show that $\{Y_{n}^{A}, \kappa_n^{A}\}_{n \geq 0}$
	 is $\GG$-equivalent to a QBD process with transition matrices 
				\begin{equation} \label{unnamedA}
				\mbox{Down} := \left[\begin{array}{cr} 
				   0 & 0 \\
				   0 & \frac{1}{2} I 
				\end{array} \right], 
				\; \mbox{Same} := \left[\begin{array}{cc} 
					0 & P_{\lambda +-} \\
					0 & \frac{1}{2} V_{\lambda}
				\end{array} \right], 
				\;  \mbox{Up} := \left[\begin{array}{cc} 
				  P_{\lambda ++} & 0 \\
				  0 & 0 
				\end{array} \right],
				\end{equation}
	with $V_{\lambda} := I + {\lambda}^{-1} U$. Indeed, the two processes have the same dynamics, apart from their behaviour in $\mathcal{S}_-$. Starting from a state in $\mathcal{S}_-$, the process $\{Y_{n}^A, \kappa_n^{A}\}$ moves one level down immediately, and moves to a state in $\mathcal{S}_-$ with transition matrix $V_{\lambda}$, whereas the process in~\eqref{unnamedA}
	\begin{itemize}
		\item makes an arbitrary number $\ell \geq 0$ of same-level transitions within $\mathcal{S}_-$, each with transition matrix $\frac12 V_\lambda$;
		\item then finally moves down one level, again within $\mathcal{S}_-$, with transition matrix $\frac12 I$.
	\end{itemize}
	However, these dynamics also produce an eventual decrease by one level with transition matrix $V_\lambda$, since
					\begin{align*} 
						\left(I - \frac{1}{\lambda}U\right)^{-1} = \frac{1}{2}\left(I - \frac{1}{2}V_{\lambda}\right)^{-1} = \sum_{\ell = 0}^{\infty} \left(\frac{1}{2}V_{\lambda}\right)^{\ell} \left(\frac{1}{2}I\right). 
					\end{align*} 
	Hence the two processes have the same $\GG$-matrix, since this matrix depends only on the probabilities of reaching a certain level and state, not on the time taken.

	We now use a similar argument to show that the process with transition matrices~\eqref{unnamedA} is $\GG$-equivalent to $\{Y_n^{\Delta}, \kappa_n^{\Delta}\}$. Indeed, these two process have the same dynamics, apart from the fact that in the latter a same-level transition with probability $\frac12 V_\lambda$ is replaced by
	\begin{itemize}
		\item either a transition to the same level \bleu{and into a state in $\mathcal{S}_{-}$} with probability $\frac{1}{2} P_{\lambda --}$, or  
		\item a transition to the set $\mathcal{S}_{+}$ and up one level with matrix $\frac{1}{2}P_{\lambda - +}$.
	\end{itemize}
	However, we have already established that, starting from a state in $\mathcal{S}_+$ at a certain level $\ell$, the process eventually reaches level $\ell-1$ in a state $\mathcal{S}_-$ with transition probability matrix $\Psi$. Hence both dynamics eventually return to the same level (for the first time), in $\mathcal{S}_-$, with the same transition probability matrix, since $V_{\lambda} = P_{\lambda --} + P_{\lambda -+} \Psi$. Again, this change does not affect the value of~$\mathcal{G}$.
	\end{proof}
					
%

Da Silva Soares and Latouche \cite[Eqn. (17)]{soares02} constructed a Quasi-Birth-Death process $\{Y^{\Theta}_n, \kappa^{\Theta}_n\}_{n \geq 0}$ with the following transition probability matrices:
\begin{align} \label{qbdguy}
\Theta_{-1} & := \begin{bmatrix}
  0 & \frac12 P_{\mu+-}\\
  0 & P_{\mu--}
\end{bmatrix},  \; 
\Theta_0 := \begin{bmatrix}
  \frac12P_{\mu++} & 0\\
  P_{\mu-+} & 0
\end{bmatrix},  \; 
\Theta_1 := \frac12\begin{bmatrix}
  I & 0\\
  0 & 0
\end{bmatrix},
\end{align}
and showed that their $\GG$-matrix is also given by~\eqref{eqn:GB-Guy}.

\begin{theorem}
	\label{theo:qbd-2}
Consider a Quasi-Birth-Death process $\{Y^{B'}_n, \kappa^{B'}_n\}_{n \geq 0}$ with transition probability matrices: 
\begin{align} 
		\label{qbdguymod}
B_{-1}' & := \frac12\begin{bmatrix}
  0 &  P_{\mu+-}\\
  0 & P_{\mu--}
\end{bmatrix},  \; 
B_0' := \frac12\begin{bmatrix}
  P_{\mu++} & 0\\
  P_{\mu-+} & I
\end{bmatrix},  \; 
B_1' := \frac12\begin{bmatrix}
  I & 0\\
  0 & 0
\end{bmatrix}.
\end{align}

The process $\{Y^{B'}_n, \kappa^{B'}_n\}_{n \geq 0}$ is equivalent to the QBD $\{Y_n^{\Theta}, \kappa_n^{\Theta}\}_{n \geq 0}$  and to the QBD $\{Y_n^{B}, \kappa_n^{B}\}$ introduced in Lemma~\ref{lem:secondrep}.

\end{theorem}

\begin{proof} This proof follows closely the one of Theorem~\ref{theo:qbd-1}.
		For the first part, note that the only difference between the dynamics of the two QBDs is in the transition matrices from a state in $\mathcal{S}_{-}$. With equal probabilities, our QBD $\{Y_{n}^{B'}, \kappa_n^{B'}\}$ either follows the same transition probabilities in their QBD, or makes a transition to the same state and the same level, that is, stays exactly where it is. Intuitively, while this slows down the expected time to go down one level, it does not affect the probability of doing so, and thus does not change the $\GG$-matrix.
	
		We now show that $\{Y_n^{B}, \kappa_n^B\}$ is $\GG$-equivalent to a QBD process with transition matrices
			\begin{equation} \label{unnamedB}
				\mbox{Down} := \left[\begin{array}{cc} 
				0 & \frac{1}{2}P_{\mu +-} \\ 
				0 & V_{\mu} 
				\end{array} \right], 
				\; \mbox{Same} := \left[\begin{array}{cc} 
				\frac{1}{2}P_{\mu ++} & 0 \\ 
				0 & 0 
				\end{array} \right], 
				\; \mbox{Up} := \left[\begin{array}{cc} 
					\frac{1}{2} I & 0 \\
					0 & 0 
				\end{array}\right].
			\end{equation}
		Indeed, the two processes have the same dynamics, apart from the fact that (starting from a state in $\mathcal{S}_+$), instead of going one level up into $\mathcal{S}_+$ with transition matrix $\left(I - \mu^{-1}T_{++}\right)^{-1}$, the new process will
		\begin{itemize}
			\item make a certain number $\ell \geq 0$ of same-level transitions, each with probability matrix $\frac12 P_{\mu ++}$, and then
			\item make an upward transition with probability matrix $\frac12 I$.
		\end{itemize}
		Nevertheless, both dynamics produce an increase by one level with the same transition probability matrix, since
		\begin{align*} 
			\left(I - \mu^{-1}T_{++}\right)^{-1} = \sum_{\ell = 0}^{\infty} \left(\frac{1}{2}P_{\mu ++} \right)^{\ell} \frac{1}{2}I. 
		\end{align*}
	Next, we show that the process~\eqref{unnamedB} is $\GG$-equivalent to $\{Y_n^{\Theta}, \kappa_n^{\Theta}\}$. The only difference between the two is that a downward transition from $\mathcal{S}_-$ to itself with transition matrix $V_{\mu}$ is replaced by
	\begin{itemize} 
		\item either a transition within $\mathcal{S}_{-}$ and down one level, with $P_{\mu --}$, or, 
		\item a transition from $\mathcal{S}_-$ to $\mathcal{S}_{+}$ at the same level, with $P_{\mu - +}$.
	\end{itemize}
	Again, starting from $\mathcal{S}_+$ the process eventually decreases by one level with transition matrix $\Psi$, so both processes produce a decrease by one level with the same transition matrix, since $V_\mu = P_{\mu--} + P_{\mu-+}\Psi$. This change does not affect the matrix $\GG$.
\end{proof}

\begin{theorem}
	\label{theo:qbd-3}
Consider a Quasi-Birth-Death process $\{Y_n^{C'}, \kappa_n^{C'}\}_{n \geq 0}$ with the following transition matrices: 
\begin{align} 
		\label{qbdgen}
C_{-1}' & := \frac12\begin{bmatrix}
  0 &  P_{\mu+-}\\
  0 & P_{\mu--}
\end{bmatrix}, \;  C_0' := \frac12\begin{bmatrix}
  P_{\mu++} & P_{\lambda+-}\\
  P_{\mu-+} & P_{\lambda--}
\end{bmatrix}, \;
C_1' := \frac12\begin{bmatrix}
  P_{\lambda++} & 0\\
  P_{\lambda-+} & 0
\end{bmatrix}.
\end{align}

Then, this process is equivalent to the QBD $\{Y_n^C, \kappa_n^C\}$; that is, its $\GG$-matrix is also given by~\eqref{eqn:GC}.
\end{theorem}


\begin{proof} In this proof we combine the transformations introduced in the previous ones. We first show that $\{Y_n^{C}, \kappa_n^C\}$ is $\GG$-equivalent to the QBD with transition matrices
	\begin{equation} \label{unnamedC} 
	\mbox{Down}:= \left[\begin{array}{cc} 
	0 & \frac{1}{2} P_{\mu {+-}} \\ 
	0 & W \end{array}\right], \; 
	\mbox{Same} := \left[\begin{array}{cc} 
	\frac{1}{2}P_{\mu ++} & \frac{1}{2}P_{\lambda +-} \\ 
	0 & 0 
	\end{array} \right], \;  \mbox{Up} := \left[\begin{array}{cc} 
	 	\frac{1}{2} P_{\lambda ++} & 0 \\
		0 & 0 
	 \end{array} \right].
	\end{equation}
We obtain this QBD from $\{Y_n^{C}, \kappa_n^C\}$ with the following modifications:
	\begin{itemize}
		\item[(a)] within the $\mathcal{S}_{+}$, we replace the transition matrix $(I - \mu^{-1} T_{++})^{-1} P_{\lambda ++}$ of increasing a level while remaining in $\mathcal{S}_{+}$ with 
		\begin{itemize} 
		\item staying in the same level and in $\mathcal{S}_{+}$ with transition matrix $1/2P_{\mu++}$ \bleu{for} any $n \geq 0$ number of steps, and then 
		 \item increas\bleu{ing} a level with $1/2P_{\lambda ++}$;
		\end{itemize} 
		\item[(b)] from $\mathcal{S}_{+}$ to $\mathcal{S}_{-}$ and decreasing a level, we replace the transition matrix $(I - \mu^{-1} T_{++})^{-1} P_{\mu +-}$ with 
		\begin{itemize} 
		\item staying in the same level and in $\mathcal{S}_{+}$ with transition matrix $1/2P_{\mu++}$ for any $n \geq 0$ number of steps, and then 
		 \item decreasing a level with $1/2P_{\mu +-}$;
		\end{itemize} 
		\item[(c)] from $\mathcal{S}_{+}$ to $\mathcal{S}_{-}$ and staying in the same level, we replace the transition matrix $(I - \mu^{-1} T_{++})^{-1} P_{\lambda +-}$, with 
		\begin{itemize} 
		\item staying in the same level and in $\mathcal{S}_{+}$ with transition matrix $1/2P_{\mu++}$ for any $n \geq 0$ number of steps, and then 
		 \item transition into a state in $\mathcal{S}_{-}$ in the same level with $1/2P_{\lambda +-}$.
		\end{itemize} 
	\end{itemize}

All these replacements produce equivalent transitions which do not alter the matrix $\GG$, since
	\begin{align*}
		\left(I - \frac{1}{\mu} T_{++}\right)^{-1}& = \left(I - \frac{1}{2}P_{\mu ++} \right)^{-1}\frac{1}{2}  = \sum_{n = 0}^{\infty} \left(\frac{1}{2}P_{\mu++} \right)^n \frac{1}{2} .
	\end{align*} 

	Next, with a similar argument we can show that~\eqref{unnamedC} is $\GG$-equivalent to the QBD with transition matrices
	\begin{equation} \label{unnamedC2} 
	\mbox{Down} := \left[\begin{array}{cc} 
	0 & \frac{1}{2} P_{\mu {+-}} \\
		\vspace*{-0.3cm} \\ 
	0 & \frac{1}{2}V_{\mu} \end{array}\right], \; 
	\mbox{Same} := \left[\begin{array}{cc} 
	\frac{1}{2}P_{\mu ++} & \frac{1}{2}P_{\lambda +-} \\ 
		\vspace*{-0.3cm} \\
	0 & \frac{1}{2}V_{\lambda} 
	\end{array} \right], \;  \mbox{Up}   := \left[\begin{array}{cc} 
	 	\frac{1}{2} P_{\lambda ++} & 0 \\
			\vspace*{-0.3cm} \\
		0 & 0 
	 \end{array} \right],
	\end{equation}
	since  
	\begin{align*} 
		W = (I - \lambda^{-1}U)^{-1}V_{\mu} = \left(I - \frac{1}{2} V_{\lambda}\right)^{-1} \frac{1}{2}V_{\mu} = \sum_{n = 0}^{\infty} \left(\frac{1}{2} V_{\lambda} \right)^n\frac{1}{2}V_{\mu}.
	\end{align*} 
	Finally, with the same technique we prove that \eqref{unnamedC2} is $\GG$-equivalent to $\{Y_{n}^{C'}, \kappa_n^{C'}\}$, since $V_{\mu} = P_{\mu --} + P_{\mu -+} \Psi$ and $V_{\lambda} = P_{\lambda --} +  P_{\lambda -+} \Psi.$
\end{proof} 

\section{Interpretation of Doubling Algorithms}
	\label{sec:int}

	We consider for the remainder of the paper the most general Quasi-Birth-Death process $\{Y_{n}^{C'}, \kappa_n^{C'}\}$ only. The same results hold for the other two processes, $\{Y_n^{A'}, \kappa_n^{A'}\}$ and $\{Y_n^{B'}, \kappa_n^{B'}\}$, by setting $\mu^{-1} := 0$ and $\lambda^{-1} := 0$, respectively. Moreover, since $\{Y_n^{A}, \kappa_n^{A}\}$, $\{Y_n^{B}, \kappa_n^{B}\}$ (as well as $\{Y_n^{\Delta}, \kappa_n^{\Delta}\}$ and $\{Y_n^{\Theta}, \kappa_n^{\Theta}\}$) are $\GG$-equivalent, the results hold for them too.
	
	For notational simplicity, we write $\{Y_{n}, \kappa_n\}$ instead of $\{Y_{n}^{C'}, \kappa_n^{C'}\}$, dropping the superscript. We also assume that the time $0$ marks the start of a busy period; hence $Y_0 = 1$, $\kappa_0 \in \mathcal{S}_{+}$. 

	\subsection{Level, mid-levels, and initial values of ADDA}
	
	Consider a sequence of integer-valued random variables $\{\tau_k\}_{k \geq 0}$ representing the times at which there is a level change, that is,
	\[
		\tau_{0} := 0, \quad \text{and} \quad \tau_{k+1} := \min \{n > \tau_{k} \colon Y_n \neq Y_{\tau_k} \} \text{ for $k > 0$}.
	\]
	We define an associated process $\{M_{\tau_k}, \kappa_{\tau_k}\}_{k \geq 0}$ as follows:  
	\begin{equation} 
		\label{def:halflevel}
		M_{\tau_k} := \begin{cases}
			Y_{\tau_k} - 1/2 & \text{if $\kappa_{\tau_k} \in \mathcal{S}_+$},\\
			Y_{\tau_k} + 1/2 & \text{if $\kappa_{\tau_k} \in \mathcal{S}_-$}.
		\end{cases}
	\end{equation}

	To see what the process $\{M_{\tau_k}, \kappa_{\tau_k}\}_{k \geq 0}$ represents in terms of the QBD $\{Y_n,\kappa_n\}$, we define \emph{mid-levels} to be the sequence $\{z + 1/2\}_{z \geq 0}$, and say that the QBD process $\{Y_n,\kappa_n\}$ \emph{crosses} the mid-level $z + 1/2$, at a certain time $n$, if its level changes from $z$ to $z+1$ or from $z+1$ to $z$. Then, the process $\{M_{\tau_k}\}$ has a simple interpretation:
	%
		it is the mid-level crossed by the process $\{Y_n, \kappa_n\}$ at times $\tau_k$. Moreover, if $Y_n$ has just completed an upward crossing at time $\tau_k$, then $\kappa_{\tau_k} \in \mathcal{S}_+$; if a downward crossing, then $\kappa_{\tau_k} \in \mathcal{S}_-$.
	%
%
\bleu{Figure~\ref{fig:halves} shows a sample path of $\{Y_n,\kappa_n\}$ and the corresponding sample path of $\{M_{\tau_k}\}$.}
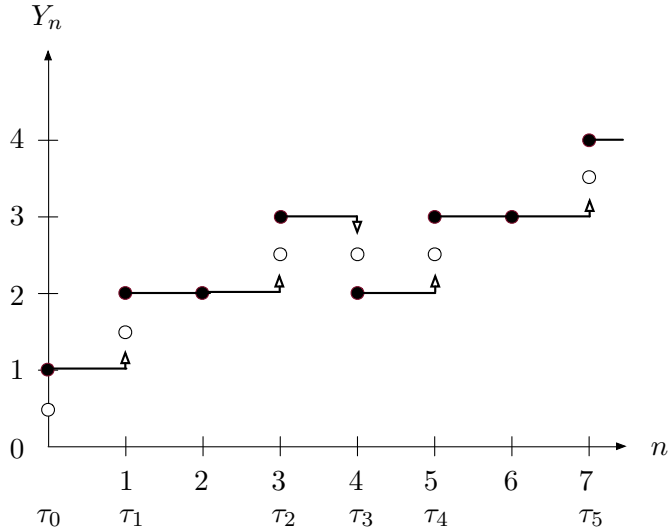
\begin{figure}[h!]
    \centering
  		\includegraphics[scale=0.4]{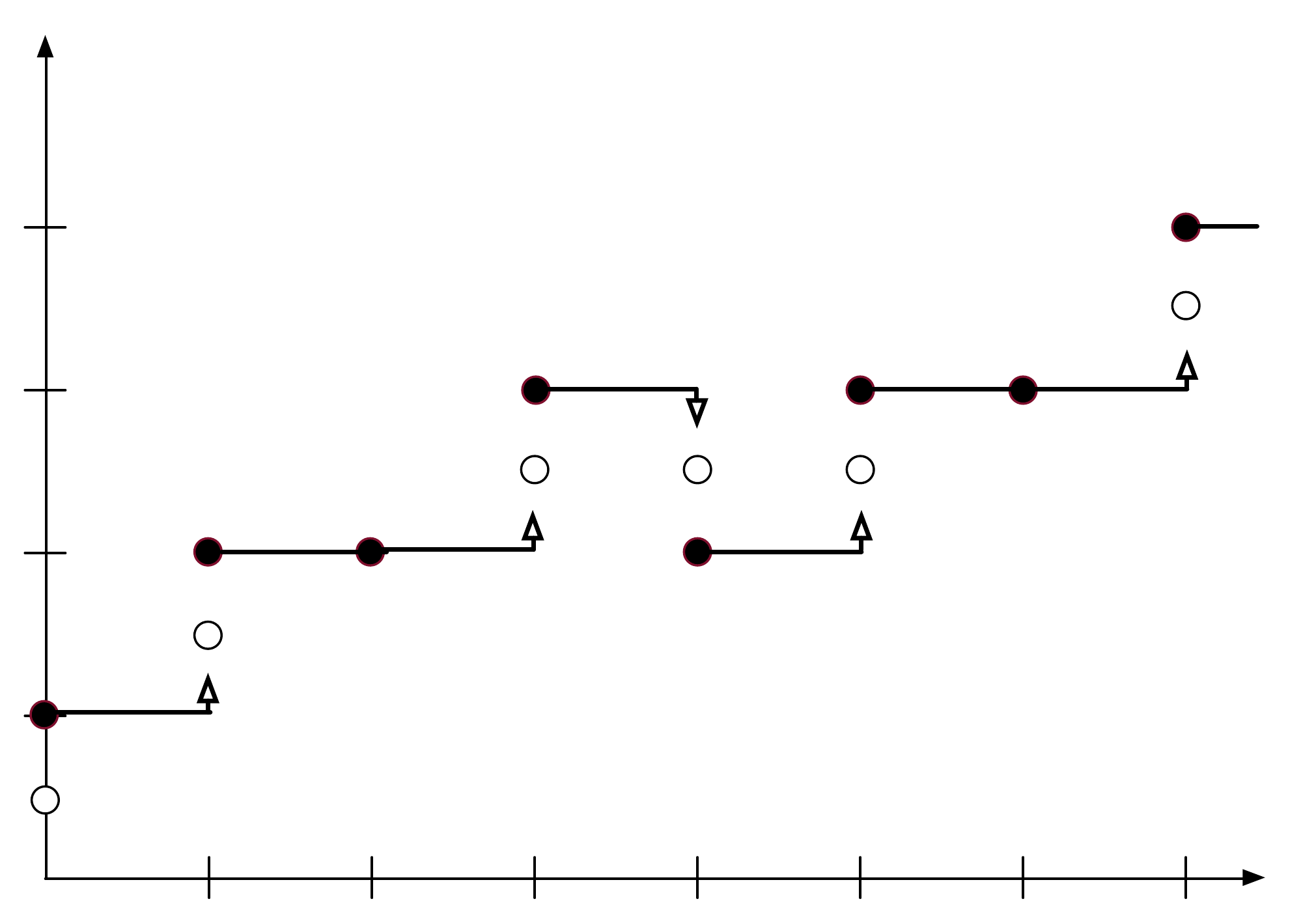}
		\put(-240,7){\makebox(0,0)[l]{$0$}} 
		\put(-240,37){\makebox(0,0)[l]{$1$}} 
		\put(-240,65){\makebox(0,0)[l]{$2$}} 
		\put(-240,95){\makebox(0,0)[l]{$3$}} 
		\put(-240,124){\makebox(0,0)[l]{$4$}} 
		\put(-199, -5){\makebox(0,0)[l]{$1$}} 
		\put(-171, -5){\makebox(0,0)[l]{$2$}} 
		\put(-142, -5){\makebox(0,0)[l]{$3$}} 
		\put(-113, -5){\makebox(0,0)[l]{$4$}} 
		\put(-85, -5){\makebox(0,0)[l]{$5$}} 
		\put(-55, -5){\makebox(0,0)[l]{$6$}} 
		\put(-27, -5){\makebox(0,0)[l]{$7$}}
		\put(-230, -19){\makebox(0,0)[l]{$\tau_0$}} 		
		\put(-199, -19){\makebox(0,0)[l]{$\tau_1$}} 
		\put(-142, -19){\makebox(0,0)[l]{$\tau_2$}} 
		\put(-113, -19){\makebox(0,0)[l]{$\tau_3$}} 
		\put(-85, -19){\makebox(0,0)[l]{$\tau_4$}} 
		\put(-27, -19){\makebox(0,0)[l]{$\tau_5$}} 
		\put(0,7){\makebox(0,0)[l]{$n$}} 
		\put(-232,170){\makebox(0,0)[l]{${Y}_n$}} 
     \caption{A sample path of $\{Y_n, \kappa_n\}$, in which $Y_n$ is represented by $\bullet$. By assumption, $\kappa_0 \in \mathcal{S}_+$. We have $\tau_1=1$, $\tau_2=3$, $\tau_3=4$, $\tau_4=5$, $\tau_5=7$, and $\kappa_{\tau_1}, \kappa_{\tau_2}, \kappa_{\tau_4}, \kappa_{\tau_5} \in \mathcal{S}_{+}, \, \kappa_{\tau_3} \in \mathcal{S}_{-}$. The sequence $\{M_{\tau_k}\}_{k \in \{0,\ldots, 5\}}$ is represented by $\circ$. \label{fig:halves}
     }
\end{figure}

	\begin{theorem}
		The process $\left\{M_{\tau_k}, \kappa_{\tau_k}\right\}_{k\in\mathbb{N}}$ is \bleu{a Quasi-Birth-Death process} on the state space $\{\mathbb{Z} + 1/2\} \times \mathcal{S}$ 
		 with transition probability matrices 
		\begin{align}
		 \label{eq:qbdD}
		 \hspace*{-0.3cm}
			D_{-1} := 
			\begin{bmatrix}
			  0 & 0\\
			  0 & F
			\end{bmatrix}, \,
			D_0 := \begin{bmatrix}
			  0 & G\\
			  H & 0
			\end{bmatrix}, \,
			D_1 := \begin{bmatrix}
			   E & 0\\
			   0 & 0
			\end{bmatrix},
		\end{align}
		where $E\in\mathbb{R}^{n_+\times n_+}, F\in\mathbb{R}^{n_-\times n_-}, G\in\mathbb{R}^{n_+ \times n_-}, H\in\mathbb{R}^{n_-\times n_+}$ are given by
		\begin{subequations}
		\label{eqn:EGHF}
		\begin{align}
			\begin{bmatrix}
			  E & G\\
			  H & F
			\end{bmatrix}
			&:=
			\left(
			I-\frac12
			\begin{bmatrix}
			  P_{\mu++} & P_{\lambda+-}\\
			  P_{\mu-+} & P_{\lambda--}
			\end{bmatrix}
			\right)^{-1}
			\frac12
			\begin{bmatrix}
			  P_{\lambda++} & P_{\mu+-}\\
			  P_{\lambda-+} & P_{\mu--}
			\end{bmatrix}
			\\
			&=
			\begin{bmatrix}
			  I - \mu^{-1}T_{++} & -\lambda^{-1}T_{+-}\\
			  -\mu^{-1}T_{-+} & I - \lambda^{-1}T_{--}
			\end{bmatrix}^{-1}
			\begin{bmatrix}
			  I+\lambda^{-1}T_{++} & \mu^{-1}T_{+-}\\
			  \lambda^{-1}T_{-+} & I + \mu^{-1}T_{--}
			\end{bmatrix}
			.
		\end{align}
		\end{subequations}
	\end{theorem}
	Note that~\eqref{eqn:EGHF} coincides with~\eqref{eqn:P0} in the case of unit rates, after setting $\mu^{-1}=\alpha$, $\lambda^{-1}=\beta$.

	\begin{proof}
	Let $J_{-1}, J_0,$ and $J_1$ represent the transition probability matrices of $\{Y_{\tau_k}, \kappa_{\tau_k}\}_{k \in \mathds{N}}$, the time-changed QBD obtained by observing the original process $\{Y_n, \kappa_n\}$ only at times at which there are changes of levels. Then, 
		\begin{align} 
			\label{eqn:Jmatrices}
			J_{-1} = (I-C_0')^{-1}C_{-1}' = \begin{bmatrix}
			  0 & G\\
			  0 & F
			\end{bmatrix}, \quad J_0 = 0, \quad J_1 
			= (I-C_0')^{-1}C_{1}' = \begin{bmatrix}
			  E & 0\\
			  H & 0
			\end{bmatrix}.
		\end{align}
		In the transitions corresponding to the sub-block $H$, the process $\{Y_{\tau_k}, \kappa_{\tau_k}\}$ moves from $(z, \mathcal{S}_-)$ at time $\tau_{\ell - 1}$ up to $(z+1, \mathcal{S}_+)$ at time $\tau_{\ell}$ for some $\ell \in \mathds{N}$. Hence, the mid-level, as defined in~\eqref{def:halflevel}, does not change from the epoch $\tau_{\ell - 1}$ to the epoch $\tau_{\ell}$, that is, $M_{\tau_{\ell - 1}} = M_{\tau_{\ell}} = z + 1/2$. Thus, for the process $\{M_{\tau_k}, \kappa_{\tau_k}\}$, the sub-block $H$ belongs to the bottom-left corner of the transition probability matrix $D_0$. 
		
		Analogous considerations hold for the other sub-blocks: $E$ corresponds to moving from a mid-level to the one above, $F$ to the one below, and $G$ corresponds to staying in the same mid-level.
	\end{proof}
	%
	
		The $\GG$-matrix of QBD~\eqref{eq:qbdD} is
		\begin{equation} \label{eqn:GD}
			\GG_D = \begin{bmatrix}
			  0 & \Psi W\\
			  0 & W
			\end{bmatrix},
		\end{equation}
		as observed in~\cite{bmp}. Indeed, decreasing from the mid-level $(z + 1/2)$ to the mid-level $(z - 1/2)$ for $\{M_{\tau_k}, \kappa_{\tau_k}\}$ corresponds to decreasing one level also for QBD $\{Y_n, \kappa_n\}$ if the initial state is in $\mathcal{S}_-$, since it means going from level $z$ to level $(z-1)$. On the other hand, such a decrease of $\{M_{\tau_k}, \kappa_{\tau_k}\}$ corresponds to a decrease of \emph{two} levels if the initial state is in $\mathcal{S}_+$, since it means going from level $(z+1)$ to level $(z-1)$.
		\begin{rem}
		The fact that this $\GG$-matrix does not contain $\Psi$ explicitly as a block is not an issue for our computation, since in doubling algorithms we obtain $\Psi$ from the limit $\Psi = \lim_{k\to\infty} G_k$, not as a block of $\GG_D$.
		\end{rem}
		\begin{rem}
		Expanding into blocks the relation $D_{-1}+D_0\GG_D+D_1\GG_D^2=\GG_D$, one obtains that $\Psi$ is a solution of the matrix equation 
		\begin{equation} \label{dare}
		X = G+EX(I-HX)^{-1}F.	
		\end{equation}
		Moreover, using the minimality of $\GG_D$, one can prove that $\Psi$ must be the minimal solution to~\eqref{dare}.
		\end{rem}

	\subsection{Interpretation of the doubling iteration} \label{subsec:doublingmap}

	Now that we have a physical interpretation for the QBD with transition matrices $D_{-1},D_0,$ and $D_1$, we consider what it means to apply one step of the iteration~\eqref{eqn:SDA} to it. 
	Bini \emph{et al.}~\cite{bmp} showed, via algebraic computations, that applying one step of the Cyclic Reduction (CR) algorithm
	\begin{subequations} \label{cr}
	\begin{align}
	\widehat{D}_{-1} &= D_{-1}(I-D_0)^{-1}D_{-1}, \label{cr:-1} \\
	\widehat{D}_0 &= D_0 + D_1(I-D_0)^{-1}D_{-1} + D_{-1}(I-D_0)^{-1}D_1,\\
	\widehat{D}_1 &= D_{1}(I-D_0)^{-1}D_1
	\end{align}
	\end{subequations}
	to the transition matrices in~\eqref{eq:qbdD} produces a QBD with the same block structure, i.e.,
	\begin{align}
	\label{crhat}
	 \hspace*{-0.3cm}
		\widehat{D}_{-1} := 
		\begin{bmatrix}
		  0 & 0\\
		  0 & \widehat{F}
		\end{bmatrix}, \,
		\widehat{D}_0 := \begin{bmatrix}
		  0 & \widehat{G}\\
		  \widehat{H} & 0
		\end{bmatrix}, \,
		\widehat{D}_1 := \begin{bmatrix}
		   \widehat{E} & 0\\
		   0 & 0
		\end{bmatrix},
	\end{align}
	with $\widehat{E},\widehat{F},\widehat{G},\widehat{H}$ precisely as in~\eqref{eqn:SDA}.
	\begin{rem}
		For more detail on Cyclic Reduction, see~\cite{blm05}; in particular, note that the full description includes a fourth equation, which is necessary in the general case to recover the matrix $\Psi$, but we shall see that in our case we can do without it. 
	\end{rem}
	By building on this relation with Cyclic Reduction, we can give a physical interpretation of the doubling iteration and justify the formulas~\eqref{eqn:SDA} directly. Our starting point is the physical interpretation of Cyclic Reduction~\cite{bmr}.
	\begin{lem}
	The matrices $\widehat{D}_{-1},\widehat{D}_0,\widehat{D}_1$ are the transition matrices of the QBD that corresponds to censoring out the odd-numbered levels $(1,3,5,\dots)$ from a QBD with transition matrices $D_{-1}, D_0, D_1$.
	\end{lem}
	\begin{rem}
	The QBD resulting from the censoring has level set $2\mathbb{Z} = \{\dots, -4,-2,0,2,4,\dots \}$, so with $\widehat{D}_{-1}$ we mean the probability of moving from level $y$ to level $y-2$, and similarly for the other two matrices. Note that we have already similarly abused the concept of levels when considering the QBD $\{M_n, \kappa_n\}$ which has level set $\mathbb{Z}+1/2 = \{\dots,-3/2,-1/2,1/2,3/2,\dots\}$.
	\end{rem}

	It is not difficult to obtain the formulas~\eqref{cr} directly from this interpretation: for instance, the formula~\eqref{cr:-1} for $\widehat{D}_{-1}$ corresponds to the probability (starting from an even-numbered level $\ell$) of descending to level $\ell-1$, spending an arbitrary amount of time there, and then taking a second step down to $\ell-2$.

	This interpretation tells us that $\widehat{D}_{-1},\widehat{D}_0,\widehat{D}_1$ are the transition matrices of the QBD obtained from $\{M_{\tau_k}, \kappa_{\tau_k}\}$ by censoring out the \emph{odd-numbered} mid-levels, i.e., those of the form $2j+3/2$, and keeping only those of the form $2j+ 1/2$, $j\in\mathbb{Z}$. More formally, define 
	\begin{align*} 
    \widehat{\tau}_0 & =0,\\ 
    \widehat{\tau}_{k+1} & = \min\left\{n > \widehat{\tau}_k : \text{$n=\tau_i$ for some $i$, and $M_n=2j+1/2$ for some $j\in\mathbb{Z}$}\right\}.
	\end{align*} 
	Then, $\widehat{D}_{-1},\widehat{D}_0,\widehat{D}_1$ are the transition matrices of $\{M_{\widehat{\tau}_k}, \kappa_{\widehat{\tau}_k}\}$, which is a QBD with level set $2\mathbb{Z}+1/2 = \{\dots,-3/2,1/2,5/2,9/2,\dots\}$.

	Armed with this physical interpretation, we can justify the formulas~\eqref{eqn:SDA} directly. To compute the probabilities $\widehat{E}$ (after starting at the mid-level $(y+1/2)$ and in a positive phase) of reaching mid-level $(y+5/2)$ \emph{before} descending to $(y+1/2)$, we condition on the number of visits to the mid-level equidistant between the two, i.e., $y+3/2$: we can visit it either 
	\begin{itemize} 
		\item once (from below, with probability matrix $E^2$), or 
		\item three times (from below, then from above, then from below, with probability matrix $EGHE$), or  
		\item five times, or 
		\item \dots, 
	\end{itemize} 
	leading to the series
	\[ 
		\widehat{E} = E^2 + EGHE + EGHGHE + \dots = E(I-GH)^{-1}E.
	\]
	To compute the complementary probabilities $\widehat{G}$ of returning to mid-level $y+ 1/2$ with taboo level $y+ 5/2$, we condition again on the number of visits to the mid-level $y + 3/2$: we can visit it either 
	\begin{itemize} 
		\item zero times (with probability $G$), or
		\item two times (from below -- from above, with probability $EGF$), or 
		\item four times, or 
		\item \dots, 
	\end{itemize} 
	leading to the series
	\[
		\widehat{G} = G + EGF + EGHGF + \dots = G + E(I-GH)^{-1}GF.
	\]
	The formulas for $\widehat{F}$ and $\widehat{H}$ are obtained analogously.

	\subsection{Mean number of visits and recurrence}

	Another matrix of physical interest for our QBD process $\{M_n,\kappa_n\}$ is the minimal solution $\RR$ of the matrix equation 
		\begin{align*} 
			\RR^2D_{-1}+\RR D_0+D_1=\RR,
		\end{align*} which represents the expected number of visits to mid-level $(1+ 1/2)$, starting from mid-level $1/2$ and before returning to the same mid-level (see~\cite[Section~5]{blm05} for a description of this matrix in the context of a general QBD). We call it the \emph{$\RR$-matrix} of the QBD, in analogy with the $\GG$-matrix. This matrix is related to its counterpart for the fluid queue: the matrix $K=T_{++}+\Psi T_{-+}$ such that $\mathrm{e}^{Kx}$ is the expected number of visits to level $x$ of the fluid process $\{\overline{X}_t, \varphi_t\}$, starting from level $0$ and before returning to it~\cite[Section~1.8]{soares05}.

	\begin{theorem}
		The $\RR$-matrix of the QBD process $\{M_n,\kappa_n\}$ is given by
		\begin{equation}  \label{eqn:R}
			\RR_D = 
			\begin{bmatrix}
			  R_1 & R_1\Psi\\0 & 0
			\end{bmatrix},			
		\end{equation}
		where
		\begin{equation} \label{eqn:R1}
			R_1 := E(I-\Psi H)^{-1} = (I-\mu^{-1} K)^{-1}(I + \lambda^{-1}K).	
		\end{equation}
	\end{theorem}
	\begin{proof}
		We start by proving the latter equality in~\eqref{eqn:R1}. From~\eqref{eqn:EGHF}, we get
		\[
			\begin{bmatrix}
			  I - \mu^{-1}T_{++} & -\lambda^{-1}T_{+-}\\
			  -\mu^{-1}T_{-+} & I - \lambda^{-1}T_{--}
			\end{bmatrix}
			\begin{bmatrix}
			  E\\H
			\end{bmatrix}=
			\begin{bmatrix}
			  P_{\lambda++}\\
			  P_{\lambda-+}
			\end{bmatrix}.
		\]
		Pre-multiplying this relation with $\begin{bmatrix}
		  I & \Psi
		\end{bmatrix}$, we obtain
		\begin{align*}
			P_{\lambda++} + \Psi P_{\lambda-+} &= (I-\mu^{-1}T_{++} -\Psi \mu^{-1}T_{-+})E + (\Psi- \lambda^{-1}T_{+-} -\lambda^{-1}\Psi T_{--} )H\\
			&= (I-\mu^{-1}T_{++} -\Psi \mu^{-1}T_{-+})E + (P_{\lambda++}+\Psi P_{\lambda-+})\Psi H,
		\end{align*}
		where the last equality follows from the algebraic Riccati equation~\eqref{eqn:Psi}. Rearranging terms in the last equation leads to
		\begin{align*}
			E(I-\Psi H)^{-1} &= (I-\mu^{-1}T_{++} -\Psi \mu^{-1}T_{-+})^{-1}(P_{\lambda++} + \Psi P_{\lambda-+}) \\
			&= (I-\mu^{-1} K)^{-1}(I + \lambda^{-1}K).
		\end{align*}

		The block structure of the matrix $\RR_D$ can be derived from the relation $\RR_D=D_1(I-D_0-D_1\GG_D)^{-1}$ (which is~\cite[Theorem 3.20, item 2]{blm05}).
	\end{proof}

	\begin{rem} \label{rem:Rs}
		Replacing $\mu^{-1}$ (resp. $\lambda^{-1}$) with $0$, one gets an expression for the $\RR$-matrix of $\{Y_n^{A'},\kappa_n^{A'}\}$ (resp. $\{Y_n^{B'},\kappa_n^{B'}\}$), and the same expression holds for the equivalent process $\{Y_n^{\Delta}, \kappa_n^{\Delta}\}$ (resp. $\{Y^{\Theta}_n, \kappa^{\Theta}_n\}$).	
	\end{rem}

	We compute the corresponding quantity $\RR_{C'}$ also for the QBD~\eqref{qbdgen}.
	\begin{theorem} \label{thm:Rc}
		The expected number of visits matrix $\RR_{C'}$ of the Quasi-Birth-Death process $\{Y_n^{C'},\kappa_n^{C'}\}$ with transition matrices~\eqref{qbdgen} is
		\[
		\RR_{C'} = 
		(I-C_0')\begin{bmatrix}
		  E & 0\\
		  H & 0
		\end{bmatrix}
		\begin{bmatrix}
		  I & R_1\\
		  0 & (I-H\Psi)^{-1}
		\end{bmatrix} ^{-1}
		(I-C_0')^{-1}		  
		\]
		and has spectral radius equal to $\rho(R_1)$.
	\end{theorem}
	\begin{proof}
	Using again the formula from~\cite[Theorem~3.20]{blm05}, we get
	\begin{align*}
		\RR_{C'} &= C_1'(I-C_0'-C_1'\GG_C)^{-1}
		\\&=
		C_1'(I-(I-C_0')^{-1}C_1'\GG_C)^{-1}(I-C_0')^{-1}
		\\&=
		(I-C_0')\begin{bmatrix}
		  E & 0\\
		  H & 0
		\end{bmatrix}
		\left(I- \begin{bmatrix}
		  E & 0\\
		  H & 0
		\end{bmatrix} 
		\begin{bmatrix}
		  0 & \Psi\\
		  0 & W
		\end{bmatrix}
		\right)^{-1}
		(I-C_0')^{-1}
		\\&=
		(I-C_0')\begin{bmatrix}
		  E & 0\\
		  H & 0
		\end{bmatrix}
		\begin{bmatrix}
		  I & -E\Psi\\
		  0 & I-H\Psi
		\end{bmatrix} ^{-1}
		(I-C_0')^{-1}
		\\&=
		(I-C_0')\begin{bmatrix}
		  E & 0\\
		  H & 0
		\end{bmatrix}
		\begin{bmatrix}
		  I & E\Psi(I-H\Psi)^{-1}\\
		  0 & (I-H\Psi)^{-1}
		\end{bmatrix}
		(I-C_0')^{-1}, 		  
	\end{align*}
	where the third equality follows from~\eqref{eqn:Jmatrices}, and the last one from the expression for the inverse of a block triangular matrix.

	Since $\rho(MN)=\rho(NM)$ for any two matrices, 
	\begin{align*} 
	\rho(\RR_{C'}) & = \rho\left(
		\begin{bmatrix}
		  I & E\Psi(I-H\Psi)^{-1}\\
		  0 & (I-H\Psi)^{-1}
		\end{bmatrix}
		\begin{bmatrix}
		  E & 0\\
		  H & 0
		\end{bmatrix}
	\right)  \\
	& = \rho\left(
	\begin{bmatrix}
	  R_1 & 0\\
	  (I-H\Psi)^{-1}H & 0
	\end{bmatrix}
	\right) \\
	& = \rho(R_1),
	\end{align*} 
	where we have used the identity $I+\Psi (I-H\Psi)^{-1}H=(I-\Psi H)^{-1}$.
	\end{proof}

	Using these results, we can show that the fluid queue and its associated QBD have the same recurrence properties. We restrict to irreducible queues here, since recurrence may not be well defined for a reducible queue.
	\begin{theorem}
	Let the rate matrix $T$ be irreducible. Then, the fluid queue $\{X_t,\varphi_t\}$, its associated QBD $\{Y_{n}, \kappa_n\}$, and the mid-level QBD $\left\{M_{\tau_k}, \kappa_{\tau_k}\right\}$ have the same recurrence properties; i.e., either all of them are transient, all of them are null recurrent, or all of them are positive recurrent.
	\end{theorem}
	\begin{proof}
	It is well established~\cite[Section~5]{blm05} that the first return matrix $\GG$ is stochastic for a recurrent QBD and substochastic for a transient one. An analogous characterization holds for a fluid queue~\cite[Theorem~4.5]{glr11b}: the matrix $U$ is a generator for a recurrent queue and a sub-generator for a transient queue. Hence we want to show that $\rho(\GG_C)=\rho(\GG_D)=\rho(W)$ equals $1$ if and only if $U$ has a $0$ eigenvalue.

	Note that, in an irreducible fluid queue, all entries of $U$ are positive, since there are paths between any two states, and the queue may remain in every state for an arbitrary long or short amount of time. Let $\gamma \leq 0$ be the Perron eigenvalue of $U$, i.e., $\bs{v} U = \gamma \bs{v}$ for a suitable row vector $\bs{v} > 0$. Then, 
	\begin{equation} \label{perronG}
		\bs{v}W = \bs{v}(I+\mu^{-1} U)(I-\lambda^{-1} U)^{-1} = \bs{v}\frac{1+\mu^{-1} \gamma}{1-\lambda^{-1} \gamma},
	\end{equation}
	which is $1$ when $\gamma=0$ and strictly smaller than $1$ when $\gamma < 0$. Since $\bs{v} > \bs{0}$, this must be the Perron eigenpair of $W$; hence, $\rho(W)=1$ when $U$ is a generator and $\rho(W)<1$ when $U$ is a subgenerator.

	The above argument proves that $\{X_t, \varphi_t\}$ is transient (resp. recurrent) whenever $\{Y_n,\kappa_n\}$ and $\{M_n,\kappa_n\}$ are transient (resp. recurrent). To tell apart null recurrent and positive recurrent processes, we need to consider the spectral radii of $\RR$ and $K$ as well.

	By an analogous argument, all entries of $K$ are positive; if $\bs{w}K=\delta \bs{w}$, with $\bs{w}> \bs{0}$ a row vector and $\delta \leq 0$, is its Perron eigenpair, then
	\begin{equation} \label{perronR}
		\bs{w}R_1 = \bs{w}(I-\mu^{-1} K)^{-1}(I + \lambda^{-1}K) = \bs{w}\frac{1 + \lambda^{-1}\delta}{1-\mu^{-1} \delta}
	\end{equation}
	is the Perron eigenpair of $R_1$. 
	Hence, the QBD is positive recurrent $\iff$ $\rho(\RR)=1$ $\iff$ $K$ has a zero eigenvalue $\iff$ the fluid queue is positive recurrent; the first and last implications follow from known properties of QBDs \cite[Section~5]{blm05} and fluid queues \cite[Theorem~4.5]{glr11b}, respectively .
	\end{proof}
	\begin{cor} \label{cor:increasing}
	The spectral radii of the $\GG$-matrix and of the $\RR$-matrix of these Quasi-Birth-Death processes are increasing functions of $\lambda$ and $\mu$.
	\end{cor}

	\subsection{Convergence properties}

	We now revisit the convergence results for doubling algorithms (Theorem~\ref{thm:doublingconvergence}) using our interpretation. An algebraic proof of this theorem is already given in~\cite{ChiCGMLX09,glx05,wwl}; the aim of this section is not to give another formal proof, but only a heuristic justification that appeals to physical arguments.

	As the matrix $\Psi$ is a block of $\GG_C$, it has a physical meaning for the QBD $\{Y_n,\kappa_n\}$, similar to the one it has for the fluid process:
	\begin{align*}
	(\Psi)_{ij} &= \mathds{P}[\text{$Y$ reaches level $y$ for the first time, in}\\
	& \;\;\;\;\; \text{phase $j\in\mathcal{S}_- \mid (Y_0,\kappa_0) = (y+1,i)$, $i\in\mathcal{S}_+$}],
	\end{align*}
	or, equivalently, in terms of mid-levels,
	\begin{align*}
	(\Psi)_{ij} &= \mathds{P}[\text{$M$ returns to mid-level $y+1/2$ for the first time, in}\\
	& \;\;\;\;\; \text{phase $j\in\mathcal{S}_- \mid (M_0,\kappa_0) = (y+1/2,i)$, $i\in\mathcal{S}_+$}].
	\end{align*}
	
	Moreover, by repeating the censoring process described in Section~\ref{subsec:doublingmap} $k$ times, we obtain
	\begin{align}
		& (E_k)_{ij} = \mathds{P}[\text{$M$  reaches mid-level $(y+2^k+1/2)$ in state $j\in\mathcal{S}_+$ \emph{before}} \nonumber \\ 
		& \quad\quad\quad\quad \text{returning to $(y+1/2)$} \mid \text{$(M_0,\kappa_0) = (y+1/2,i)$, $i\in\mathcal{S}_+$} ],\\
		& (F_k)_{ij} = \mathds{P}[\text{$M$ reaches mid-level $(y-2^k+1/2)$ in state $j\in\mathcal{S}_-$ \emph{before}}  \nonumber \\ 
		& \quad\quad\quad\quad \text{returning to $(y+1/2)$} \mid \text{$(M_0,\kappa_0) = (y+1/2,i)$, $i\in\mathcal{S}_-$} ], \label{eqn:Fks5} \\
		& (G_k)_{ij} = \mathds{P}[\text{$M$ returns to mid-level $(y+1/2)$ in state $j\in\mathcal{S}_-$ \emph{before}}  \nonumber \\
		& \quad\quad\quad\quad \text{reaching $(y+2^k+1/2)$} \mid \text{$(M_0,\kappa_0) = (y+1/2,i)$, $i\in\mathcal{S}_+$} ],\\
		& (H_k)_{ij} = \mathds{P}[\text{$M$ returns to mid-level $(y+1/2)$ in state $j\in\mathcal{S}_+$ \emph{before}}  \nonumber \\
		& \quad\quad\quad\quad \text{reaching $(y-2^k+1/2)$} \mid \text{$(M_0,\kappa_0) = (y+1/2,i)$, $i\in\mathcal{S}_-$} ].
	\end{align}
	In particular, from the expression for $G_k$ it follows that $$0 \leq G_0 \leq G_1 \leq G_2 \leq \dots,$$ and comparing it to the expressions for $\Psi$ one gets
	\begin{align*}
	&(\Psi - G_k)_{ij} = \mathds{P}[\text{$M$ returns to mid-level $(y+1/2)$ in state $j\in\mathcal{S}_-$ \emph{after}} \\ 
		& \;\;\;\;\; \text{\emph{reaching $(y+2^k+1/2)$ at least once}} \mid \text{$(M_0,\kappa_0) = (y+1/2,i)$, $i\in\mathcal{S}_+$} ].
	\end{align*}

	We use these formulas in three different arguments, depending on the recurrence character of the process.
	
	If the QBD $\{M_n,\kappa_n\}$ is transient, then:  
	\begin{itemize} 
		\item[(a)] The probability of reaching mid-level $(y-2^k+1/2)$ starting from mid-level $(y+1/2)$ --- i.e., descending $2^k$ levels --- is given by the entries of $\GG_D^{2^k}$, which go to zero like $\rho(W)^{2^k}$. 
		\item In~\eqref{eqn:Fks5}, entries of $F_k$ are expressed as the sum of the probabilities of certain sample paths involving descending from mid-level $(y+1/2)$ to $(y-2^k+1/2)$. Hence, these entries all converge to zero as $O(\rho(W)^{2^k})$. 
		\item Similarly, the entries of $\Psi - G_k$ are given by the sum of probabilities of certain sample paths, all of them involving descending from mid-level $(y+2^k+1/2)$ to $(y+1/2)$. Hence, they converge to zero as $O(\rho(W)^{2^k})$.
	\end{itemize} 
	
	If the QBD $\{M_n,\kappa_n\}$ is positive recurrent, then:  
	\begin{itemize} 
		\item[(a)] Symmetrically, the probability of reaching mid-level $(y+2^k+1/2)$ starting from $(y+1/2)$ --- i.e., ascending $2^k$ levels --- goes to zero as $O(\sigma^{2^k})$ for a certain $\sigma$. By considering the level-reversed process, one can show that $\sigma=\rho(\RR_D)=\rho(R_1)$, see e.g.~\cite[Theorem~5.8]{lr93}.
		\item[(b)] The entries of $E_k$ and those of $\Psi-G_k$ are given by the sum of probabilities of certain sample paths, all of them involving ascending from mid-level $(y+1/2)$ to mid-level $(y+2^k+1/2)$. Hence, they all go to zero as $O(\rho(R_1)^{2^k})$.
	\end{itemize} 
	
	If the QBD $\{M_n,\kappa_n\}$ is null recurrent, then:
	\begin{itemize} 
		\item[(a)] Intuitively, moving up and down are equally likely (when suitably averaged over long times). So the probability of reaching mid-level $y+2^k+1/2$ starting from $(y+1/2)$ before returning to $(y+1/2)$ behaves like the analogous probability for a simple one-dimensional random walk, i.e., it goes to zero like $O(2^{-k})$. This shows that the entries of $E_k$ (and, symmetrically, those of $F_k$) go to zero as $O(2^{-k})$. 
		\item[(b)] Similarly, the entries in $\Psi - G_k$ are associated to sample paths that involve reaching mid-level $(y+2^k+1/2)$ before returning to $(y+1/2)$, hence they also go to zero as $O(2^{-k})$.
	\end{itemize}
	
	Reversing the up and down directions corresponds to swapping $E_k$ with $F_k$ and $G_k$ with $H_k$. Hence, in particular, $H_k$ converges to the matrix $\widehat{\Psi}$, which is the analogue of $\Psi$ for the level-reversed process. 

	\section{Conclusions} \label{sec:conclusion}
	We conclude our journey through the doubling algorithms with a brief discussion of its advantages. There are two main improvements with respect to Cyclic Reduction (CR) on the various QBDs introduced in Section~\ref{sec:rei}:
	\begin{itemize}
		\item The transition matrices~\eqref{eq:qbdD}, obtained by switching to the mid-level QBD $\{M_{\tau_k}, \kappa_{\tau_k}\}$, have only four nonzero blocks, and this zero structure is preserved by CR/doubling, which yields at each step QBDs of the form
		\begin{align}
		 \label{eq:crDi}
		 \hspace*{-0.3cm}
			D^{(i)}_{-1} = 
			\begin{bmatrix}
			  0 & 0\\
			  0 & F_i
			\end{bmatrix}, \,
			D^{(i)}_0 = \begin{bmatrix}
			  0 & G_i\\
			  H_i & 0
			\end{bmatrix}, \,
			D^{(i)}_1 = \begin{bmatrix}
			   E_i & 0\\
			   0 & 0
			\end{bmatrix}.
		\end{align}
		In contrast, six blocks fill in during the steps of CR on the QBDs from Section~\ref{sec:rei}. This is the main reason why the cost per step of doubling is smaller: there are fewer matrices to store and update.
		\item In the positive recurrent (resp. transient) case, the convergence speed of both doubling and Cyclic Reduction depends on the spectral radius of the $\RR$-matrix (resp. $\GG$-matrix) of the involved QBD: the smaller this spectral radius is, the faster the convergence. This property is shown in~\cite[Theorem~7.6]{blm05} for CR. The spectral radius is an increasing function of $\lambda$ and $\mu$ (Corollary~\ref{cor:increasing}).

		Hence, in addition to each iteration being faster, ADDA with $\mu^{-1}=\alpha=\alpha_{\mathrm{opt}}$, $\lambda^{-1}=\beta = \beta_{\mathrm{opt}}$ requires fewer iterations to reach convergence than the algorithms based on CR on~\eqref{eqn:qbdram}--\eqref{qbdguymod}, which correspond to choosing $\mu=\infty$ or $\lambda=\infty$.


	\end{itemize}
	
	\section*{Acknowledgement} The first two authors would like to acknowledge the support of ACEMS (ARC Centre of Excellence for Mathematical and Statistical Frontiers). The third author would like to acknowledge the support of INDAM (Istituto Nazionale di Alta Matematica) and of a PRA 2017 project of the university of Pisa. All authors are grateful to Oscar Peralta for his comments on a preliminary version of this work.

\bibliography{interpretation}

\begin{thebibliography}{10}

\bibitem{AhnR03}
S.~Ahn and V.~Ramaswami.
\newblock Fluid flow models and queues---a connection by stochastic coupling.
\newblock {\em Stoch. Models}, 19(3):325--348, 2003.

\bibitem{asmussen95}
S.~Asmussen.
\newblock Stationary distributions for fluid flow models with or without
  {B}rownian noise.
\newblock {\em Communications in {S}tatistics: {S}tochastic {M}odels},
  11(1):21--49, 1995.

\bibitem{bean05}
N.~Bean, M.~M. O'Reilly, and P.~G. Taylor.
\newblock Algorithms for return probabilities for stochastic fluid flows.
\newblock {\em Stoch. Models}, 21(1):149--184, 2005.

\bibitem{blm05}
D.~A. Bini, G.~Latouche, and B.~Meini.
\newblock {\em Numerical Methods for Structured {M}arkov Chains}.
\newblock Numerical Mathematics and Scientific Computation. Oxford University
  Press, Oxford, 2005.

\bibitem{bm09}
D.~A. Bini and B.~Meini.
\newblock The {C}yclic {R}eduction algorithm: from {P}oisson equation to
  stochastic processes and beyond.
\newblock {\em Numer. Algorithms}, 51:23--60, 2009.

\bibitem{bmp}
D.~A. Bini, B.~Meini, and F.~Poloni.
\newblock Transforming algebraic {R}iccati equations into unilateral quadratic
  matrix equations.
\newblock {\em Numer. Math.}, 116(4):553--578, 2010.

\bibitem{bmr}
D.~A. Bini, B.~Meini, and V.~Ramaswami.
\newblock A probabilistic interpretation of {C}yclic {R}eduction and its
  relationships with {L}ogarithmic {R}eduction.
\newblock {\em Calcolo}, 45(3):207--216, 2008.

\bibitem{ChiCGMLX09}
C.-Y. Chiang, E.~K.-W. Chu, C.-H. Guo, T.-M. Huang, W.-W. Lin, and S.-F. Xu.
\newblock Convergence analysis of the doubling algorithm for several nonlinear
  matrix equations in the critical case.
\newblock {\em SIAM J. Matrix Anal. Appl.}, 31(2):227--247, 2009.

\bibitem{soares05}
A.~da~Silva~Soares.
\newblock {\em Fluid queues -- Building upon the analogy with QBD processes}.
\newblock PhD thesis, Universit\'{e} libre de Bruxelles, 2005.

\bibitem{soares02}
A.~da~Silva~Soares and G.~Latouche.
\newblock Further results on the similarity between fluid queues and {QBDs}.
\newblock In G.~Latouche and P.~Taylor, editors, {\em Matrix-analytic methods:
  {T}heory and applications}, pages 89--106. World Scientific, Singapore, 2002.

\bibitem{glr11b}
M.~Govorun, G.~Latouche, and M.-A. Remiche.
\newblock Stability for fluid queues: characteristic inequalities.
\newblock {\em Stoch. Models}, 29(1):64--88, 2013.

\bibitem{guo01}
C.-H. Guo.
\newblock Nonsymmetric algebraic {R}iccati equations and {W}iener-{H}opf
  factorization for {M}-matrices.
\newblock {\em {SIAM} Journal on Matrix Analysis and Applications},
  23(1):225--242, 2001.

\bibitem{glx05}
X.-X. Guo, W.-W. Lin, and S.-F. Xu.
\newblock A structure-preserving doubling algorithm for nonsymmetric algebraic
  {R}iccati equation.
\newblock {\em Numer. Math.}, 103:393--412, 2006.

\bibitem{kk95}
R.~L. Karandikar and V.~Kulkarni.
\newblock Second-order fluid flow models: {R}eflected {B}rownian motion in a
  random environment.
\newblock {\em Oper. Res}, 43:77--88, 1995.

\bibitem{lr93}
G.~Latouche and V.~Ramaswami.
\newblock A {L}ogarithmic {R}eduction algorithm for {Q}uasi-{B}irth-{D}eath
  processes.
\newblock {\em J. Appl. Prob.}, 30:650--674, 1993.

\bibitem{lr99}
G.~Latouche and V.~Ramaswami.
\newblock {\em Introduction to Matrix Analytic Methods in Stochastic Modeling}.
\newblock ASA-SIAM Series on Statistics and Applied Probability. SIAM,
  Philadelphia PA, 1999.

\bibitem{NguP15}
G.~T. Nguyen and F.~Poloni.
\newblock Componentwise accurate fluid queue computations using doubling
  algorithms.
\newblock {\em Numer. Math.}, 130(4):763--792, 2015.

\bibitem{ram99}
V.~Ramaswami.
\newblock Matrix analytic methods for stochastic fluid flows.
\newblock In D.~Smith and P.~Hey, editors, {\em Teletraffic Engineering in a
  Competitive World (Proceedings of the 16th International Teletraffic
  Congress)}, pages 1019--1030. Elsevier Science B.V., Edinburgh, UK, 1999.

\bibitem{roger94}
L.~C.~G. Rogers.
\newblock Fluid models in queueing theory and {Wiener-Hopf} factorization of
  {Markov} chains.
\newblock {\em Ann. Appl. Probab.}, 4:390--413, 1994.

\bibitem{wwl}
W.-G. Wang, W.-C. Wang, and R.-C. Li.
\newblock Alternating-directional doubling algorithm for {$M$}-matrix algebraic
  {R}iccati equations.
\newblock {\em SIAM J. Matrix Anal. Appl.}, 33(1):170--194, 2012.

\end{thebibliography}
\bibliographystyle{abbrv}

\end{document}